\documentclass[psamsfonts]{amsart}

\usepackage{amssymb,amsfonts}
\usepackage[all,arc]{xy}
\usepackage{enumerate}
\usepackage{mathrsfs}
\usepackage{amsmath}
\usepackage{fullpage}
\usepackage{hyperref}
\usepackage{graphicx}

\newtheorem{theorem}{Theorem}

\newtheorem{claim}{Claim}

\newtheorem{corollary}{Corollary}

\newtheorem{lemma}{Lemma}

\newtheorem{proposition}{Proposition}
\newtheorem{remark}{Remark}

\numberwithin{equation}{section}

\title{Asymptotically conical Ricci-flat K\"ahler metrics on $\mathbb{C}^2$ with cone singularities along a complex curve}
\author{Martin de Borbon}
\begin{document}
\maketitle
\begin{abstract}
We prove an existence theorem for asymptotically conical Ricci-flat K\"ahler metrics on $\mathbb{C}^2$ with cone singularities along a smooth complex curve. These are expected to arise as blow-up limits of non-collapsed sequences of K\"ahler-Einstein metrics with cone singularities.
\end{abstract}

\section{Introduction}

We work on $\mathbb{C}^2$ with standard complex coordinates $z, w$. Let $P=P(z, w)$ be a degree $d \geq 2$ polynomial such that $C= \lbrace P=0 \rbrace$ is a smooth complex curve. Write
$$P = P_d + Q, $$
with $P_d$ the  homogeneous degree $d$ part of $P$ and $\deg (Q) \leq d-1$. We restrict to the case when $C$ has $d$ different asymptotic lines, which means that the zero locus of $P_d$ consists of $d$ distinct complex lines $L_1, \ldots, L_d$. Denote $L= \cup_{k=1}^d L_k$.

We fix a number $\beta$ such that
\begin{equation} \label{restriction on beta}
\frac{d-2}{d}< \beta < 1 .
\end{equation}
It follows from the work of Troyanov \cite{Troyanov1} and Luo-Tian \cite{LuoTian} that for \(\beta\) in this range there is a unique compatible spherical metric \(g\) on $\mathbb{CP}^1$ with cone angle $2 \pi \beta$ at the points corresponding to \(L\). The metric $g$ lifts, through the Hopf map, to a polyhedral K\"ahler cone metric $g_F$ on $\mathbb{C}^2$ with cone angle \(2\pi\beta\) along the complex lines \(L\) -see \cite{Panov}-. As a Riemannian cone, we have \(g_F= dr^2 + r^2 \overline{g}\) with $\overline{g}$ a constant sectional curvature \(1\) metric on the \(3\)-sphere with cone angle \(2\pi\beta\) in transverse directions to the Hopf circles determined by $L$.

We let $\mathcal{D}$ be the set of all diffeomorphisms $H$ of $\mathbb{C}^2$ for which there is a compact set $K$ such that $H(C\setminus K) \subset L$, we also require \(H\) to be  asymptotic to the identity in the following sense:  There are constants $A_j$ such that $|H(x) -x| \leq A_0$,  $|DH(x) -Id| \leq A_1|x|^{-1}$ and $|D^{\alpha} H(x)| \leq A_j |x|^{-j}$ for all $x \in \mathbb{C}^2$ and $j = |\alpha| \geq2$. It is elementary to prove that $\mathcal{D}$ is not empty. Our main result is the following

\begin{theorem} \label{thm}
There is a  K\"ahler metric $g_{RF}$ on $\mathbb{C}^2$ with cone angle $2\pi\beta$ along $C$ and $H \in \mathcal{D}$ such that

\begin{itemize}
\item 
\begin{equation} \label{RF}
\omega^2_{RF} = |P|^{2\beta-2} \Omega \wedge \overline{\Omega}, 
\end{equation}
where $\omega_{RF}$ is the associated K\"ahler form and $\Omega= (1/\sqrt{2}) dz  dw$.
		
\item 
\begin{equation} \label{AC}
|(H^{-1})^{*}g_{RF} - g_F|_{g_F} \leq  A r^{\nu}
\end{equation}
outside a compact set,  for some constants $A>0$ and $\nu < 0$. 
\end{itemize}

\end{theorem}
A foremost precedent to Theorem \ref{thm} is Section 5 in \cite{Donaldson} on model Ricci-flat metrics with cone singularities along the conic $\lbrace zw=1 \rbrace$, of the kind we consider when $d=2$. The metrics in \cite{Donaldson} are invariant under the $S^1$-action $e^{i\theta}(z, w) = (e^{i\theta}z, e^{-i\theta}w)$, are almost-explicit and are constructed by means of the Gibbons-Hawking ansatz. In contrast, we use PDE methods and $g_{RF}$ is far from being explicit. A few remarks regarding Theorem \ref{thm} are in order
\begin{enumerate}
	\item The metric \(g_{RF}\) is smooth on the complement of the curve and has cone singularities along \(C\) in a H\"older continuous sense as defined in \cite{Donaldson}. We will recall this notion later on, meanwhile we limit to say that if \((z_1, z_2)\) are local complex coordinates in which \(C=\{z_1=0\}\)  and \(g_{(\beta)}\) is the local model metric
	\begin{equation*}
	g_{(\beta)} = \beta^2 |z_1|^{2\beta-2} |dz_1|^2 + |dz_2|^2;
	\end{equation*} 
	then there is $\lambda >1$ such that \(\lambda^{-1} g_{(\beta)} \leq g_{RF} \leq \lambda g_{(\beta)} \) in a neighbourhood of $\{z_1=0\}$. 
	
	\item The  Ricci form of $g_{RF}$ is given by $-i\partial\overline{\partial} \log \det (g_{RF})$. It follows from equation \ref{RF} that, up to a constant factor,  $\det (g_{RF})$ is equal to $|P|^{2\beta -2}$. Since $P$ is holomorphic and  non-vanishing in the complement of the curve we conclude that $g_{RF}$ is Ricci-flat on its smooth locus.
	
	\item The proof shows that the asymptotic behaviour \ref{AC} holds in a stronger $C^{\alpha}$ sense and it also specifies the asymptotic rate \(\nu\).
	
\end{enumerate}

Our interest in Theorem \ref{thm} comes from the blow-up analysis of the K\"ahler-Einstein (KE) equations in the context of solutions with cone singularities. 
In the case of \emph{smooth} KE metrics on complex surfaces the solutions can degenerate -in the non-collapsed regime- only by developing isolated orbifold points -see \cite{Anderson}-, and the blow-up limits at these are the well-known ALE spaces -see \cite{Kronheimer}-. In the conical case a new feature arises when the curves along which the metrics have singularities degenerate. In this setting, Theorem \ref{thm} furnishes  models for blow-up limits at a point where a sequence of smooth curves develops an ordinary \(d\)-tuple point.

\subsection*{Outline of proof of Theorem \ref{thm} and content of the paper}
In a few words we can say that our proof uses the continuity method and goes along the lines of Yau's proof of the Calabi conjecture \cite{Yau}. The work of Yau has been extended to the context of metrics with cone singularities by  Brendle \cite{Brendle} and Jeffres-Mazzeo-Rubinstein \cite{JMR} -among others- and to the context of ALE and AC metrics by Joyce \cite{Joyce} and Conlon-Hein \cite{ConlonHein} -among others again-. Our work mixtures  \cite{Brendle}, \cite{JMR} and \cite{Joyce}, \cite{ConlonHein}.

In Section \ref{Fmetrics} we review Panov's polyhedral K\"ahler cones \cite{Panov}. In particular, we provide an explicit formula for the potential of \(g_F\) in terms of the corresponding spherical metric on the projective line. We can think of \(g_F\) as an `approximate solution' to \ref{RF}, in the sense that if \(\omega_F\) denotes its K\"ahler form then \(\omega_F^2= |P_d|^{2\beta-2} \Omega \wedge \overline{\Omega}\).

In Section \ref{Rmetrics} we construct $H \in \mathcal{D}$ and a reference metric $\omega$ (we use the common abuse of language which undistinguishes a metric from its K\"ahler form). The reference metric \(\omega\) has cone angle $2\pi \beta$ along $C$ and is asymptotic to $\omega_{F}$.  In Subsection \ref{upbound} we construct another metric,  $\omega_{B}$, which is quasi-isometric to $\omega$ and has the fundamental property that its bisectional curvature is bounded from above.  The proof of this bound goes along the lines of Appendix A in \cite{JMR}. We finish Section \ref{Rmetrics} by establishing a Sobolev inequality  for our reference metrics.

We develop the relevant linear theory in  Section \ref{la}. First we review some foundational material from \cite{Donaldson}. We recall Donaldson's interior Schauder estimates, which are of fundamental importance in our analysis. Having the interior estimates at hand, in Subsections \ref{wh} and \ref{mr} we  set up a theory of  weighted H\"older spaces. Our  references are \cite{Bartnik}, Chapter 2 in \cite{PacardRiviere} and Chapter 8 in \cite{Szekelyhidi}.  The main result of Section \ref{la} is Proposition \ref{LTHRM}, which establishes Fredholm mapping properties for the Laplacian acting in weighted spaces. This parallels  known results in the case of asymptotically conical smooth metrics, as stated in Theorem 2.11 of Conlon-Hein \cite{ConlonHein}.
In subsection \ref{ap},  as an application of Proposition \ref{LTHRM} together with the Implicit Function Theorem,  we show the existence of a metric $\omega_0$ asymptotic to $\omega_F$ such that
\begin{equation} \label{t=0} 
\omega_0^2 = e^{-f_0} |P|^{2\beta -2} \Omega \wedge \overline{\Omega} , 
\end{equation}
with $f_0$ a smooth function of compact support. What will be important for us, apart from the fact that \(\omega_0\) solves \ref{RF} outside a compact set, is that $\omega_0$ has uniformly bounded Ricci curvature. 

To prove Theorem \ref{thm} it is enough to show that there exists $u \in C^{2, \alpha}_{\delta}$ (our notation for the weighted H\"older spaces) such that
$$ (\omega_0 + i \partial \overline{\partial} u)^2 = e^{f_0} \omega_0^2 .$$
The solution is then $\omega_{RF} = \omega_0 + i \partial \overline{\partial} u$. In order to solve the equation we use Yau's continuity path and consider the set 
\begin{equation} \label{cpath}
T = \lbrace t \in [0,1] : \exists \hspace{1mm} u_t \in C^{2, \alpha}_{\delta} \hspace{3mm}\mbox{solution of} \hspace{3mm} (\omega_0 + i \partial \overline{\partial} u_t)^2 = e^{tf_0} \omega_0^2 \rbrace . 
\end{equation}
We want to prove that $1 \in T$. Proposition \ref{LTHRM} implies that $T$ is open and $0 \in T$ trivially with $u_0 =0$. The closedness of $T$ follows from the a priori estimate $\| u_t \|_{2, \alpha, \delta} \leq C$ for some constant $C>0$  independent of $t \in T$. This is the content of Proposition \ref{a priori est}, the main result of Section \ref{APRIORI},  which we prove  into several steps. We follow Joyce \cite{Joyce} closely, our main difference from \cite{Joyce} is in the \(C^2\)-estimate. The details go as follows: 
First we estimate the $C^0$-norm of $u$, to do this we use the Sobolev inequality (for the metric $\omega_0$) and then we run a Moser iteration, same as \cite{Joyce}. To estimate the $C^2$-norm of $u$ we use the maximum principle and the Chern-Lu inequality (in a slightly different way than in \cite{JMR}). Here it is crucial that we have an upper bound on the bisectional curvature of $\omega_{B}$ and a lower bound on the Ricci curvature of  $\omega_t = \omega_0 + i \partial \overline{\partial} u_t $ in the form of $\mbox{Ric}(\omega_t) \geq - A \omega_{B}$,  for some $A>0$. This bound holds for $\omega_0$ by \ref{t=0} and it holds for $\omega_t$ since along the continuity path \ref{cpath} 
\begin{equation}
\mbox{Ric}(\omega_t) = (1-t) \mbox{Ric}(\omega_0) . 
\end{equation}
The $C^2$-estimate gives us the uniform bound $ C^{-1} \omega_0 \leq \omega_t \leq C \omega_0$. Then we can apply the interior $C^{2, \alpha}$-estimate given by Theorem 1.7 of Chen-Wang \cite{ChenWang}. We proceed to the weighted estimates in Subsection \ref{WESTIMATES}. We start by proving a bound on $\|u_t\|_{0, \mu}$ for some $\delta< \mu < 0$. The technique is Moser iteration and we follow \cite{Joyce} again. The bound on $\|u_t\|_{2, \alpha, \delta}$ follows from the linear theory developed. Theorem \ref{thm} is finally proved in \ref{PFTHM}. 

\subsection*{Acknowledgment} Theorem \ref{thm} is the main outcome of the author's PhD Thesis at Imperial College, founded by the European Research Council Grant 247331 and defended in December 2015. I wish to thank my supervisor, Simon Donaldson, for his continuous support and patience.

\section{Flat metrics} \label{Fmetrics}

\subsection{Spherical metrics with cone singularities on $\mathbb{CP}^1$ } \label{SPH MET}

The expression of a spherical metric in geodesic polar coordinates $(\rho, \theta)$ around a conical point of total angle $2\pi \beta$ is 
\begin{equation} \label{spherical1}
d\rho^2 + \beta^2 \sin^2(\rho) d\theta^2.
\end{equation}
There is an induced complex structure on a punctured neighbourhood of the origin given by an anti-clockwise rotation of  angle  $\pi /2$ with respect to \ref{spherical1}. It is a basic fact that we can change coordinates so that this complex structure extends smoothly to $0$. Indeed, if we write $ \eta = \left( \tan (\rho/2) \right)^{1/\beta} e^{i\theta}$ our model metric takes the form
\begin{equation} \label{spherical2}
4\beta^2 \frac{|\eta|^{2\beta-2}}{(1+|\eta|^{2\beta})^2} |d\eta|^2 .
\end{equation}

Let $ L_1, \ldots, L_d \in \mathbb{CP}^1$ be $d$ distinct points. We want to define the notion of a compatible spherical metric $g$ with a cone singularity of angle $2\pi\beta$ at the given points. There are two equivalent points of view

\begin{itemize}
	
	\item $g$ is a metric on the $2$-sphere minus $d$ points which is locally isometric to the round sphere of radius $1$. Around each of the singular points there are polar coordinates $(\rho, \theta)$ such that $g$ is given by $\ref{spherical1}$. The metric  $g$ endows the punctured sphere with the complex structure of $\mathbb{CP}^1 \setminus \lbrace L_1, \ldots, L_d \rbrace$.
	
	\item $g$ is a compatible metric on  $\mathbb{CP}^1 \setminus \lbrace L_1, \ldots, L_d \rbrace$ of constant Gaussian curvature equal to $1$  . Around each singular point we can find a complex coordinate $\eta$ in which $g$ is given by $\ref{spherical2}$.
	
\end{itemize}

A basic result, due to Troyanov \cite{Troyanov1} and Luo-Tian \cite{LuoTian}, asserts that for $\beta$ in the range \ref{restriction on beta} there is a unique compatible spherical metric $g$ with cone angle \(2\pi\beta\) at \(L_1, \ldots, L_d\). The situation is far much complicated if we allow \(\beta>1\), uniqueness fails -see \cite{Troyanov2}- and there is not a completely understood existence theory -see \cite{MondelloPanov}-. For this, among many others reasons, we restrict all along this paper to \(\beta<1\).

When $d=3$ the metric $g$ is obtained by doubling a spherical equilateral triangle $T$ with interior angles equal to $\beta \pi$. The restriction that \(\beta>1/3\) in this case is due to the fact that the sum of the angles must be bigger than \(\pi\). More generally we can double a spherical polygon with \(d\) vertices, but when \(d \geq 4\) most spherical metrics are not of this kind. The fact is that any such $g$ fits isometrically in the round $3$-sphere as the boundary of a spherical convex polytope and this includes the doubling of polygons as degenerate cases where all the vertices lie in a totally geodesic $2$-sphere. It follows from elementary spherical geometry that the total area of $g$ is equal to \(4\pi c\) with 
\begin{equation} \label{Number c}
	c = \frac{1}{2} (2+d\beta-d).
\end{equation}  
Note that our restriction on \(\beta\), \ref{restriction on beta}, is equivalent to \(0<c<1\).

\subsection{Spherical  metrics with cone singularities  on the 3-sphere}
Write $\mathbb{R}^4 = \mathbb{R}^2 \times \mathbb{R}^2$ and take polar coordinates $ (r_1, \theta_1), (r_2, \theta_2)$ on each  factor. Consider the product of a standard cone of total angle $2\pi\beta$ with an Euclidean plane 
\begin{equation*}
g_{(\beta)} = dr_1^2 + \beta^2 r_1^2 d\theta_1^2 +  dr_2^2 +  r_2^2 d\theta_2^2 .
\end{equation*}
We want to write $g_{(\beta)}$ as a Riemannian cone. It is a general fact that the product of two metric cones is a metric cone, in our case this amounts to check that if we define $ r \in (0, \infty) $ and $ \rho  \in (0, \pi/2) $ by 
$$ r_1 = r \sin \rho, \hspace{5mm} r_2 = r \cos \rho ; $$
then $g_{(\beta)} = dr^2 + r^2  \overline{g}_{(\beta)}$ with
\begin{equation} \label{singsphere}
\overline{g}_{(\beta)}= d\rho^2 + \beta^2 \sin^2(\rho) d\theta_1^2 + \cos^2(\rho) d\theta_2^2 . 
\end{equation}
We think of $\overline{g}_{(\beta)}$ as a metric on the 3-sphere with a cone singularity of angle $2\pi \beta$ transverse to the circle given by the intersection of $ \lbrace 0 \rbrace \times \mathbb{R}^2$ with the unit sphere. 

Let $S^3 = \lbrace |z|^2 + |w|^2 = 1 \rbrace \subset \mathbb{C}^2$ equipped with the $S^1$-action $e^{it}(z, w)=(e^{it}z, e^{it}w)$ and let $H : S^3 \to \mathbb{CP}^1$ be the Hopf bundle. Denote by  $g$  the compatible metric on $\mathbb{CP}^1$ with constant curvature $K_g =4$ and cone angle $2\pi\beta$ at the points corresponding to $L$. \emph{Note that this  is $1/4$ times the spherical metrics we considered in Subsection \ref{SPH MET}.} We lift \(g\) to a spherical metric on \(S^3\) using a suitable connection.

\begin{lemma} \label{3sphere}
	
	There is an $S^1$-invariant metric $\overline{g}$ on $S^3 \setminus L$ such that
	
	\begin{itemize} 
		\item $H : (S^3 \setminus L, \overline{g}) \to (\mathbb{CP}^1 \setminus L, g)$ is a Riemannian submersion with geodesic fibres of constant length.
		\item $\overline{g}$ is locally isometric to the round 3-sphere of radius 1. 
		\item Each $p \in L$ has a neighborhood in which $\overline{g}$ agrees with  $\overline{g}_{(\beta)}$.
		
	\end{itemize}
	
\end{lemma}

Recall that \(H\) is a Riemannian submersion from the round \(3\)-sphere of radius \(1\) to \(\mathbb{CP}^1\) with its Fubini-Study metric. The last is identified via stereographic projection with \(S^2(1/2)\), the round \(2\)-sphere of radius \(1/2\). If \( \Omega \subset S^2(1/2) \) is a contractible domain then the universal cover of $ H^{-1}(\Omega) \subset S^3(1)$ is diffeomorphic to \( \Omega \times \mathbb{R} \) and its inherited constant curvature $1$ metric is invariant under translations in the \( \mathbb{R} \) factor. The planes orthogonal to the fibres define a horizontal distribution, hence a connection $\nabla$ on \( \Omega \times \mathbb{R} \). The holonomy of $\nabla$ along a closed curve \( \gamma \subset \Omega \) is equal to the parallel translation by twice the algebraic area bounded by $\gamma$. On the other hand for any $l >0$ we can take the quotient of \( \Omega \times \mathbb{R} \) by \( l \mathbb{Z} \) to obtain a metric $\overline{g}$ of constant curvature $1$ on \( \Omega \times S^1 \) such that all the fibres are geodesics of length $l$. 
Given the metric $g$ on $S^2$ with cone singularities and Gaussian curvature $4$; we can cut $S^2$ along geodesic segments with vertices at all the conical points and obtain a contractible polygon $P$ which can be immersed -by its enveloping map- in $S^2(1/2)$. Consider the metric $\overline{g}$ on \( P \times S^1 \) with $l = 2 \mbox{Area}(P)$. The holonomy of the fibration along the border of $P$ is trivial -as it makes one full rotation- and the gluing of the boundary $P$ which recovers $g$ can be lifted -using a parallel section- to a gluing of \( P \times S^1 \) to obtain the metric $\overline{g}$ of Lemma \ref{3sphere}. The area of \(P\) is equal to \(\pi c\) with \(c\) given by \ref{Number c} and the circle fibres have length \(2\pi c\).

We make the above construction (due to Panov) a bit more explicit and write \(\overline{g}\) in terms of \(g\).  
W.l.o.g. we can  assume that  $ L_j = \lbrace z = a_j w \rbrace$ with $a_j \in \mathbb{C}$  for $j = 1, \ldots , d-1$ and $ L_d= \lbrace w=0 \rbrace$. Set $\xi = z/w$, write $g = e^{2\phi} |d\xi|^2$ with $\phi$ a function of $\xi$ and let $u = \phi -  (\beta-1) \sum_{j=1}^{d-1} \log |\xi - a_j|.$ On the trivial $S^1$-bundle $\mathbb{C} \setminus \lbrace a_1, \ldots , a_{d-1} \rbrace \times S^1 \cong S^3 \setminus L$ with coordinates $(\xi, e^{it})$ we set
	\begin{equation*}
	\alpha = dt + \frac{i}{2c} (\partial u - \overline{\partial} u) . 
	\end{equation*}
Then $\alpha$ is a connection on the Hopf bundle with mild singularities along \(L\)  and satisfies the following conditions:  
\begin{itemize}
	\item  $ d\alpha = (1/2c) H^{*}( K_g dV_g),$ where \(K_g\) and \(dV_g\) denote the Gaussian curvature and area form of \(g\). 
	\item  If $p \in \mathbb{CP}^1$ is a point in $L$ and $\gamma_{\epsilon}$ is a loop that shrinks to $p$ as $\epsilon \to 0$, then the holonomy of $\alpha$ along $\gamma_{\epsilon}$ goes to the identity as $\epsilon \to 0$.
\end{itemize}
The metric of Lemma \ref{3sphere} is
	\begin{equation} \label{metric s3}
	\overline{g} = g + c^2 \alpha^2 . 
	\end{equation}
We note that \(\overline{g}\) as well as \(\alpha\) are uniquely determined up to pulling-back by a gauge transformation.

\subsection {Polyhedral K\"ahler cones}
Let \(\overline{g}\) be given by \ref{metric s3}. On $ (0, \infty) \times \mathbb{C} \setminus \lbrace a_1, \ldots, a_{d-1} \rbrace \times S^1 $ with coordinates $(r, \xi, e^{it})$ set
\begin{equation*}
g_F = dr^2 + r^2 \overline{g} .
\end{equation*}
Write $ \xi = x + iy$. There is a natural Hermitian almost-complex structure given by
$$ I \tilde{  \frac{\partial}{\partial x}  } = \tilde{  \frac{\partial}{\partial y}  } , \hspace{15mm}  I   \frac{\partial}{\partial r}   = \frac{1}{cr}   \frac{\partial}{\partial t}      $$
where
$$  \tilde{  \frac{\partial}{\partial x}  } =  \frac{\partial}{\partial x}  - \alpha \left( \frac{\partial}{\partial x} \right) \frac{\partial}{\partial t} , \hspace{15mm} \tilde{  \frac{\partial}{\partial y}  } =  \frac{\partial}{\partial y}  - \alpha \left( \frac{\partial}{\partial y} \right) \frac{\partial}{\partial t}  $$
are the horizontal lifts of $\partial / \partial x$ and $\partial / \partial y$. Finally set $\omega_F = g_F (I. , .)$. It is straightforward to check that $ \left( (0, \infty) \times \mathbb{C} \setminus \lbrace a_1, \ldots, a_{d-1} \rbrace \times S^1    , g_F, I \right) $ is a K\"ahler manifold and 
	\begin{equation*} 
	\omega_F = \frac{i}{2} \partial \overline{\partial} r^2 .
	\end{equation*}
	
The following lemma follows by straightforward computation

\begin{lemma} The functions 	
	\begin{equation*} 
	z= \xi w, \hspace{4mm} w = c^{1/2c}  r^{1/c} e^{u/2c}  e^{it}
	\end{equation*}
give a biholomorphism between $ (0, \infty) \times \mathbb{C} \setminus \lbrace a_1, \ldots, a_{d-1} \rbrace \times S^1$ with the complex structure $I$ and $\mathbb{C}^2 \setminus L$. Moreover	
	\begin{equation*} 
	\omega_F^2 = | P_d |^{2\beta - 2} \Omega \wedge \overline{\Omega} .
	\end{equation*}
\end{lemma}

We have two natural systems of coordinates: the complex coordinates $(z,w)$ and the spherical coordinates $(r,\theta)$,  where $\theta$  denotes a point in the 3-sphere. For $\lambda >0$ define $D_{\lambda} (r , \theta) = (\lambda r, \theta)$ and $m_{\lambda} (z, w) = (\lambda z, \lambda w)$. Then $D_{\lambda} = m_{\lambda^{1/c}}$ and  $m_{\lambda}^{*} \omega_F = \lambda^{2c} \omega_F$.
Starting from the complex coordinates (as we do in the paper) we let $\xi = z/w$, the asymptotic lines are $ L_j = \lbrace z = a_j w \rbrace$ with $a_j \in \mathbb{C}$  for $j = 1, \ldots , d-1$ and $ L_d= \lbrace w=0 \rbrace$, $g = e^{2\phi} |d\xi|^2$, $u = \phi -  (\beta-1) \sum_{j=1}^{d-1} \log |\xi - a_j|$ and we set 
\begin{equation} \label{form1}
r^2 = \frac{1}{c} |w|^{2c} e^{-u} .
\end{equation}
The function \(r^2\) is smooth on \(\mathbb{C}^2 \setminus L\) and extends continuously to \(\mathbb{C}^2\), it is a K\"ahler potential for the Riemannian cone \(g_F\) and $r$ measures the intrinsic distance to the apex of the cone -located at $0$-. The Reeb vector field \(I(r \partial / \partial r ) \) generates the $S^1$-action $e^{it}(z, w) = (e^{it/c}z, e^{it/c}w)$. The metric is flat on the complement of \(L\) and for every $ 0 \neq p \in L$ we can find holomorphic coordinates $(z_1, z_2)$ on a neighbourhood $U$ around $p$ such that $U \cap L= \lbrace z_1=0 \rbrace$ and $g_F$ agrees with the model  $g_{(\beta)}$
		
\section{ Reference Metrics} \label{Rmetrics}

\subsection{A  diffeomorphism} \label{D}
	
Let  $C= \lbrace P = 0 \rbrace$.  The homogeneous degree $d$ part of $P$ is $ P_d = l_1 \ldots l_d$. We write $l_j = z- a_jw$, for  $j=1, \ldots, d-1$ and $l_d = w$.  W.l.o.g. let us assume that $a_j \not= 0$ for all $j=1, \ldots, d-1$. First we look at the piece of $C$ which is asymptotic to $L_d = \lbrace w =0 \rbrace$.
	
	\begin{lemma} \label{diffeo 1}
		There exist $R, \delta > 0$ and $ \Phi = \Phi (z) : \lbrace |z|>R \rbrace \to \mathbb{C}$ bounded holomorphic, which depend only on $P$,  such that
		$$ C \cap U_{d, \delta, R} = \lbrace \left(z, \Phi(z)\right) \rbrace , $$
		where $U_{d, \delta, R}= \lbrace |w| < \delta |z|, \hspace{3mm} |z|>R \rbrace$.
	\end{lemma}

	\begin{proof}
		For $j=1, \ldots, d-1$  let $S_j$  be an  orthogonal linear transformation that takes $L_d$ to $L_j$.   Write $U_{j; \delta, R} = S_j (U_{d ; \delta, R})$ and $U_{\delta, R} = \cup_{j=1}^d U_{j; \delta, R}$. Taking $\delta$ small enough we can assume that the sets $U_{j, \delta, R}$ are pairwise disjoint. Write
		\begin{equation} \label{pol}
		P = P_d + Q
		\end{equation}
		with $Q$ a polynomial of degree at most $d-1$. On the complement of $U_{\delta, R}$ we have that $ | P_d (x)| \geq C_1 |x|^d$ for some $C_1 > 0$. Since $\deg(Q) \leq d-1$ we can find $C_2 >0$ such that $|Q(x)| \leq C_2 |x|^{d-1}$. It follows that for $R$ big enough
		\begin{equation}\label{intersection}
		C \cap \lbrace |z| > R \rbrace \subset U_{\delta, R} .
		\end{equation}
		
		For each $z$ with $|z| >R$ we write 
		\begin{equation} \label{facto}
		P(z, w)=P_{z}(w) = a (w- h_1 (z)) \ldots (w-h_d (z)) .
		\end{equation}
		With $a =(-1)^{d-1}a_1 \ldots a_{d-1} \not= 0$ and $h_j : \lbrace |z|>R \rbrace \to \mathbb{C}$ holomorphic. 
		It follows from \ref{intersection} that for each $j$,  $ \lbrace (z, h_j(z)), \hspace{2mm} |z|>R \rbrace \subset U_{i, \delta, R}$ for some $i =i(j)$. In particular this implies that there is a constant $A>0$ such that
		\begin{equation} \label{bdh}
		|h_j(z)| \leq A |z|
		\end{equation}
		for $j =1, \ldots,  d$.
		We want to show that we can label the functions $h_j$ in a way such  that $i(j)=j$. First we note that if $i(j_0) = d$ then $h_{j_0}$ is bounded.  Indeed $ | l_1 \ldots l_{d-1} (x) | \geq c |x|^{d-1}$ for some $c>0$ and all $x \in U_{d, \delta, R}$, so that $|h_{j_0} (z)|= | Q|/| l_1 \ldots l_{d-1}| \leq C_2/c$. 
		From  \ref{facto} we get that the coefficient in front of $w$ in the polynomial $P_z(w)$ is given by 
		\begin{equation} \label{lcoef}
		(-1)^{d-1}a \sum_{j=1}^d \Pi_{i \not= j} h_i(z) .
		\end{equation}
		On the other hand \ref{pol} and $ P_d = w (z-a_1 w) \ldots (z-a_{d-1}w)$, imply that \ref{lcoef} is a polynomial of degree $d-1$  in $z$ (with leading term $z^{d-1}$). If we had $ i(j_0) = i(j_1)=d$ for some $j_0 \not= j_1$ then $h_{j_0}$ and $h_{j_1}$ would be bounded. This together with the bound \ref{bdh} would imply that the absolute value of \ref{lcoef} would be bounded by a constant times $|z|^{d-2}$, contradicting \ref{lcoef} being a degree $d-1$ polynomial.
		
		Changing coordinates we can argue the same way for the other asymptotic lines. We conclude that  the map $j \to i(j)$ is injective and we can perform the desired labeling. The lemma follows by setting $\Phi = h_d$.
		In fact $h_j(z) = (1/a_j) z + \phi_j (z)$ with $\phi_j$ bounded for $j=1, \ldots d-1$ so that \ref{facto} gives
		\begin{equation} \label{FAC}
		P(z, w) = (l_1 + \phi_1  ) \ldots (l_{d-1} + \phi_{d-1})(w - \Phi)
		\end{equation} 
		
	\end{proof}
	
Recall that we denote by $\mathcal{D}$ the set of diffeomorphisms $H$ of $\mathbb{C}^2$ which, outside a compact set, map the curve $C$ to the asymptotic lines $L$ and are asymptotic to the identity in the following sense: $H(x)=x +h(x)$, with $D^{\alpha}h(x)= O(|x|^{-|\alpha|})$ for any multi-index $\alpha$. We show that \(\mathcal{D}\) is not empty, moreover we can take the diffeomorphism to be holomorphic in a suitable neighbourhood at infinity of the curve. 
	
	\begin{lemma} \label{diffeo 2}
		Let $\delta>0$ be small enough and $R>0$ big enough, then there exists a diffeomorphism $H \in \mathcal{D}$  such that $ H$ is holomorphic in $U_{\delta/2, 2R}$ and $H$ is the identity outside $U_{\delta, R}$.
	\end{lemma}
	
	\begin{proof}
		Let $\chi = \chi (t)$ be a smooth cut-off function with $\chi(t) = 1$ for $t\leq1$ and $\chi(t)=0$ for $t\geq2$. We first define $H$ in the region asymptotic to $L_d$. Let
		\begin{equation}
		h(z, w) =\chi \left( \frac{2|w|}{\delta |z|} \right) (1- \chi) (R^{-1} |z|) . 
		\end{equation}
		It follows that $h=1$ on $U_{d, \delta/2, 2R}$, $h=0$ outside $U_{d, \delta, R}$ and $ |D^{\alpha} h (x)| \leq C_{|\alpha|} |x|^{-|\alpha|}$ for any multi-index $\alpha$. We set
		\begin{equation}
		H_d (z, w) = (z, w- h \Phi) .
		\end{equation} 
		Since $\Phi$ is a bounded holomorphic function of $z$ and in the region $U_{d, \delta, R}$ we have $|z| \geq c |(z, w)|$ for some $c>0$, we conclude that  there are constants $A_j$ such that $|H_d(x) -x| \leq A_0$ ,  $|DH_d(x) -Id| \leq A_1|x|^{-1}$ and $|D^{\alpha} H_d(x)| \leq A_j |x|^{-j}$ for all $x \in \mathbb{C}^2$ and $j = |\alpha| \geq2$.
		We proceed similarly for the other asymptotic regions, and in an obvious notation we set
		\begin{equation*}
		H = H_1 \circ \ldots \circ H_d .
		\end{equation*}	
	\end{proof}

From now on we fix $\delta, R >0$ and $H$ as in Lemma \ref{diffeo 2}.

\subsection {Construction of $\omega$} \label{REF} 

First we derive some easy consequences of the scaling property of the K\"ahler potential of \(g_F\)
	\begin{equation} \label{Scale}
	r^2 \circ m_{\lambda} = \lambda^{2c} r^2 .
	\end{equation}
Since $\omega_F$ is positive we can find $\kappa_1>0$ such that $\omega_F \geq \kappa_1 \omega_{euc}$ on the euclidean unit sphere, by scaling we have that for every $p \in \mathbb{C}^2$
	\begin{equation} \label{low a}
	\omega_F (p) \geq \kappa_1 |p|^{2c-2} \omega_{euc} .
	\end{equation} 
On the other hand it follows from the continuity of $r$ that there is \(\kappa_2 >0\) such that 
	\begin{equation} \label{r}
	\kappa_2^{-1} |p|^c \leq r(p) \leq \kappa_2 |p|^c . 
	\end{equation}
Differentiating  equation \ref{Scale} on $\mathbb{C}^2 \setminus L$ we get that $D^{\alpha}r \circ m_{\lambda} = \lambda^{2c-| \alpha |} D^{\alpha} r$ for any multi-index $\alpha$. For $\epsilon >0$  denote ${U}_{\epsilon} = U_{\epsilon, 0}$, with the notation as in the proof of Lemma \ref{diffeo 1}. From the smoothness of $r$ on the complement of $L$ it follows that
	\begin{equation} \label{deriv}
	|D^{\alpha} r^2 (p) | \leq A |p|^{2c-|\alpha|}
	\end{equation} 
on $\mathbb{C}^2 \setminus {U}_{\epsilon}$, where the constant $A$  depends on $\epsilon$ and $|\alpha|$. It follows from \ref{deriv} and \ref{low a} that in the complement of $U_{\epsilon}$ there exist $\kappa_{\epsilon} > 1$ such that 
	\begin{equation} \label{compmet}
	\kappa_{\epsilon}^{-1} |p|^{2c-2} \omega_{euc} \leq \omega_F (p) \leq  \kappa_{\epsilon} |p|^{2c-2} \omega_{euc} .
	\end{equation}
	
Let us denote by $I$ the standard complex structure on $\mathbb{C}^2$ and let $G$ be the inverse of $H$.
	
	\begin{lemma}
		\begin{equation}
		|G^{*} I - I |_{g_F} = O (r^{-1/c}) .
		\end{equation}
	\end{lemma} 
	
	\begin{proof}
		First we note that $|G^{*} I - I |_{g_{euc}} = O (|p|^{-1})$,  since $ G^{*} I - I $  is  basically given by $\overline{\partial} G$. From \ref{r} we can replace $O(|p|^{-1})$ with $O(r^{-1/c})$. Secondly, there exists $\epsilon>0$ such that -outside a compact set- $G$ is holomorphic in $U_{2\epsilon}$. So $G^{*} I = I$ in $U_{2\epsilon}$.
		In a vector space with an inner product the norm of an endomorphism doesn't change if we scale the inner product by a positive constant. Hence $|G^{*} I - I |_{|p|^{2c-2}g_{euc}} = O (|p|^{-1})$. Finally \ref{compmet} gives the lemma.
		
	\end{proof}
	
	We move on and define
	\begin{equation} \label{ETA}
	\eta =\frac{ i}{2} \partial \overline{\partial} (r^2 \circ H) .
	\end{equation}

	\begin{lemma} \label{eta}
		There exists a compact $K$ such that $\eta >0$ outside $K$. Moreover,
		$$ | G^{*}\eta -  \omega_F |_{g_F} = O(r^{-1/c}) . $$
		
	\end{lemma}

	\begin{proof}
		Denote $H(z, w) = (u, v)$, so that $r^2 = r^2 (u, v)$. Write $U= U_{\delta, R}$ and $U'=U_{\delta/2, 2R}$, the subsets introduced in Lemma \ref{diffeo 2}.
		We remove compact sets whenever necessary. Note that $G^{*} \eta = \omega_F$ in  $H(U')$, clearly we can pick $\epsilon >0$ such that $ {U}_{\epsilon} \subset H(U')$. In $\mathbb{C}^2 \setminus H(U')$ we are then able to use the bounds \ref{deriv}.  Set $p_0= (u_0, v_0) = H(x_0)$ with $x_0 =(z_0, w_0) \notin U'$. First we  compute $\eta(x_0)$
		
		$$ \frac{\partial}{\partial z} (r^2 \circ H) = \frac{\partial r^2}{\partial u} \frac{\partial u}{\partial z} + \frac{\partial r^2}{\partial \overline{u}} \frac{\partial \overline{u}}{\partial z} + \frac{\partial r^2}{\partial v}  \frac{\partial v}{\partial z}     + \frac{\partial r^2}{\partial \overline{v}} \frac{\partial \overline{v}}{\partial z}  ,     $$
		
		$$ \frac{\partial^2}{\partial z \partial \overline{z}} (r^2 \circ H) = \frac{\partial^2 r^2}{\partial^2 u} \frac{\partial u}{\partial \overline{z}} \frac{\partial u}{\partial z}       +    \frac{ \partial^2 r^2}{\partial u \partial \overline{u}}    \frac{\partial \overline{u}}{\partial \overline{z}} \frac{\partial u}{\partial z}  +   \frac{ \partial^2 r^2}{\partial u \partial v}            \frac{\partial v}{\partial \overline{z}}  \frac{\partial u}{\partial z}        +   \frac{ \partial^2 r^2}{\partial u \partial \overline{v}}        \frac{\partial \overline{v}}{\partial \overline{z}}            \frac{\partial u}{\partial z} +  \frac{\partial r^2}{\partial u}   \frac{ \partial^2 u}{\partial z \partial \overline{z}}       $$
		
		$$ + ( \ldots ) ,  $$
		where $(\ldots)$ consists of 15 terms that the reader can figure out. The second term is equal to 
		$$\frac{ \partial^2 r^2}{\partial u \partial \overline{u}} (p_0)  \left(1 + O(|x|^{-1}) \right) .$$ 
		The first, third and fourth terms can be bounded by $ A |x|^{2c-2} |x|^{-1}$ and the fifth by $A|x|^{2c-1} |x|^{-2}$ for some constant $A>0$. It is easy to see that the remaining 15 terms can be bounded by  $ A |x|^{2c-2} |x|^{-1}$ (the ones which contain second derivatives of $r^2$) or $A|x|^{2c-1} |x|^{-2}$ (the ones which contain second derivatives of $H$). We conclude that we can bound all this terms by a constant times $|x|^{2c-3}$.
		We argue similarly for the other derivatives in $\partial \overline{\partial} (r^2 \circ H)$ to conclude that
		$$ G^{*} \eta (p_0) = \omega_F (p_0) + O(|x|^{2c-3}) dz d\overline{z} +  O(|x|^{2c-3}) dz d\overline{w}  + O(|x|^{2c-3}) dw d\overline{z} + O(|x|^{2c-3}) dw d\overline{w}  $$
		Note that $dzd\overline{z} = dud\overline{u} + \nu $ where $\nu$ is a 2-form with $|\nu|_{euc} = O(|p|^{-1})$. From \ref{low a} we get $|dud\overline{u}|_{g_F} = O (|x|^{2-2c})$. We argue equally for the other terms to conclude that
		\begin{equation} \label{Dcay}
		| G^{*} \eta - \omega_F |_{g_F} (p_0) = O(|p_0|^{-1}) .
		\end{equation}
		The result follows from \ref{r}.
		
	\end{proof}
	
	\begin{remark}
		
		As we already said, $ G^{*}(\eta) = \omega_F$ on a region $U_{\delta', R'}$ for some $\delta', R' >0$. In the complement of this region one can extend \ref{Dcay} to
		\begin{equation}
		| \nabla^i (G^{*} \eta - \omega_F) |_{g_F} (p_0) = O(r^{-(1/c) -i} ) , 
		\end{equation}
		where $\nabla$ is the Levi-Civita connection of $g_F$.
	\end{remark}
	
	Let $h$ be a cut-off function with $h=1$ on $B_N$ (the euclidean ball of radius $N$, say) and $h=0$ on $B_{N+1}^c$ where  $N$ is large enough so that $ C \cap B_N^c \subset U'$ and $\eta >0$ outside $B_N$. Consider 
	\begin{equation} \label{PR}
	\omega' = \frac{i}{2} \partial \overline{\partial} \left(h|P|^{2\beta} + (1-h)(r^2 \circ H) \right) .
	\end{equation}
	Note that $\omega' = \eta >0$ on $B_{N+1}^c$. On the other hand 
	$$ \omega' =\frac{i}{2} \partial \overline{\partial} |P|^{2\beta}= \beta^2 |P|^{2\beta -2} \frac{i}{2} \partial P \wedge \overline{\partial P} \geq 0$$ 
	on $B_N$. Finally consider the annulus $B_{N+1} \setminus B_N $ 
	
	\begin{claim}
		There is $a>0$ such that $\omega' \geq -a \omega_{euc}$ on $B_{N+1} \setminus B_N$.
	\end{claim}
	
	\begin{proof}
		Indeed, for $ x \in C \cap (B_{N+1} \setminus B_N)$ we can find holomorphic coordinates $(z_1, z_2)$ such that $C= \lbrace z_1 =0 \rbrace$ and $ r^2 \circ H = |z_1|^{2\beta} + |z_2|^2$. In these coordinates $ P= f z_1$ for some non-vanishing holomorphic $f$. Then we have $2\omega' = i \partial \overline{\partial} \left( h(|f|^{2\beta}|z_1|^{2\beta}) + (1-h) (|z_1|^{2\beta} + |z_2|^2) \right)=$ (smooth) $+ i\partial \overline{\partial} u$,  where $ u= |z_1|^{2\beta} ( h|f|^{2\beta} + 1 -h )$. On a smaller neighborhood we can assume $|f|^{2\beta} \geq \epsilon >0$ so that $ i\partial \overline{\partial} u = i u \partial \log u \wedge \overline{\partial} \log u + u i \partial \overline{\partial} \log u \geq u  i\partial \overline{\partial} \log F$ where $ F = h|f|^{2\beta} + 1 -h$. Note that $F$ is smooth and $F \geq \min \lbrace \epsilon, 1 \rbrace$ to conclude the claim.
	\end{proof}

	\begin{lemma} \label{AM}
		There is a K\"ahler metric $\omega$ on $\mathbb{C}^2$ with cone singularities of angle $2 \pi \beta$ along $C$ such that $\omega = \eta$ outside a compact set.
	\end{lemma}
	
	\begin{proof}
		Let $\chi = \chi (t)$ be a smooth cut-off function with $\chi(t) = 1$ for $t\leq1$ and $\chi(t)=0$ for $t\geq2$. For $L>0$ and $x \in \mathbb{C}^2$ let $\chi_L (x) = \chi (L^{-1} |x|)$. Set $\phi = \log (1+ |z|^2 + |w|^2)$ and define 
		\begin{equation}
		\omega_L = \omega' + i \Gamma \partial \overline{\partial} (\chi_L \phi)
		\end{equation}
		with  $\Gamma>0$ such that $\Gamma i\partial \overline{\partial} \phi + \omega' >0$ and $L>N+2$.   If $L$ is big enough we can assume that on the annulus on $ B_{2L} \setminus B_L$, $\omega' = \eta$. Recall that $|\eta|_{euc} \geq C_1 |x|^{c-2}$ on the other hand, on   $ B_{2L} \setminus B_L$ we can bound $ | \partial \overline{\partial} (\chi_L \phi)|_{euc} \leq C_2  |x|^{-2} \log |x|$ (with $C_2$ independent of $L$). Taking $L$ large  we get that $\omega_L$ is positive everywhere. Fix such a large $L$ and define $\omega = \omega_L$. The statement about the cone singularities follows from Lemma \ref{CSing} in Section \ref{la}.
		
	\end{proof}

	We look at the volume form of $\omega$ and define $f$ by means of the equation
	\begin{equation} \label{volom}
	\omega^2 = e^f |P|^{2\beta-2} \Omega \wedge \overline{\Omega} .
	\end{equation}
	
	\begin{lemma} \label{V}
		Outside a compact set $f$ is a smooth function with 
		\begin{equation} \label{ricpot}
		|D^{\alpha} f(x) | \leq A_{|\alpha|} |x|^{-1 - |\alpha|} .
		\end{equation}
	\end{lemma}
	
	\begin{proof}
		Consider first the complement of $U_{\delta, R}$ where $H$ is the identity and $\eta = \omega_F$, therefore
		$$ e^f = |P|^{2-2\beta} |P_d|^{2\beta-2} = \left| 1 + \frac{Q}{P_d} \right|^{2-2\beta} . $$
		In the complement of $U_{\delta, R}$  we have constants $b_{|\alpha|}$ such that
		$$ |D^{\alpha}(Q/P_d)|(x) \leq b_{|\alpha|} |x|^{-1-|\alpha|} .$$ 
		\ref{ricpot} then follows from $ f = (2-2\beta) \log |1+ Q/P_d|$.
		Secondly we consider the region $U_{\delta/2, 2R}$, where $H$ is holomorphic and $ \eta = H^{*} \omega_F$. We see that $e^f = |P/(P_d \circ H)|^{2-2\beta}$. We focus in  $U_{d, \delta/2, 2R}$ and use \ref{FAC} to get $ P/(P_d \circ H) = (1 + \psi_1 (z)) \ldots (1 + \psi_{d-1} (z))$ where $\psi_j (z)$ are holomorphic with $|\psi_j(z)| \leq A |z|^{-1}$ for some $A>0$. Note that in $U_{d, \delta/2, 2R}$ we have $|z| \geq a |(z, w)|$ for some $a>0$. As before we get \ref{ricpot}.
		Finally consider the region $U_{\delta, R} \setminus U_{\delta/2, 2R}$. By Lemma \ref{eta} we can write $ \eta = H^{*} \omega_F + \xi$ where $ \xi$ is a 2-form with $|\xi|_{g_F} = O(|x|^{-1})$. We conclude that $\eta^2 = \left( 1 + O(|x|^{-1}) \right) H^{*} \omega_F^2$ and we can proceed as before.
		
	\end{proof}
	
	We call $\omega$  our reference metric. We shall need a metric with bisectional curvature bounded from above. The author was not able to prove that $\omega$ has this property. To remedy this we introduce another metric, $\omega_B$, that we define next.

	\subsection{Upper bound on Bisec($\omega_{B})$} \label{upbound}
	
	Fix $0<\delta < 2c$. Note that the function $ p \to|p|^\delta$ is plurisubharmonic in $\mathbb{C}^2$. In fact, 
	$$ a^{-1} |p|^{\delta -2} \omega_{euc} \leq i \partial \overline{\partial} |p|^{\delta} \leq  a |p|^{\delta -2} \omega_{euc} $$
	for some $a>0$. The diffeomorphism $H$ is asymptotic to the identity, so  there is  $K>0$ such that $a^{-1} |p|^{\delta -2} \omega_{euc} \leq i \partial \overline{\partial} |H|^{\delta} \leq  a |p|^{\delta -2} \omega_{euc}$  outside a ball of radius $K$. In the construction of $\omega$ (Lemma \ref{AM}) we take $L>>K$. Let $\psi = \psi(t)$ be a smooth convex function of one real variable which is equal to the identity for large values of $t$ and is constant when $t \leq K$.  Define $ h = \psi \circ |H|^{\delta}$, then $h$ is smooth and  $\nu= i\partial \overline{\partial} h = \psi'' i \partial |H|^{\delta} \wedge \overline{\partial} |H|^{\delta} + \psi' i\partial \overline{\partial} |H|^{\delta} $, since the first term is non-negative we have that $\nu \geq 0$ in all of $\mathbb{C}^2$. Moreover,   outside a compact set there is $a>1$ such that
	$$a^{-1} |p|^{\delta-2} \omega_{euc} \leq \nu (p) \leq a |p|^{\delta-2} \omega_{euc} .$$
	We define $\omega_{B}$ as
	\begin{equation} \label{ref}
	\omega_{B} = \omega +  \Lambda\nu , 
	\end{equation}
	where $\Lambda>0$ will be specified later on. From its definition it follows that 
	\begin{equation}
	Q_2^{-1} \omega \leq \omega_{B} \leq Q_2 \omega
	\end{equation}
	for some $Q_2>0$.
	The goal is to prove the following 
	
	\begin{lemma} \label{UBD}
		$$ Bisec (\omega_{B}) \leq Q_1 .$$
	\end{lemma}

	We recall the definition of bisectional curvature. Let $\omega$ be a K\"ahler metric on an open subset $U$ of $\mathbb{C}^2$. For $x \in U$ and  $v, w \in T_x^{1,0} \mathbb{C}^2$ with $ |v|_{\omega}= |w|_{\omega} =1$ we set
	$$ \mbox{Bisec}_{\omega} (v, w) = R (v, \overline{v}, w, \overline{w}) ,  $$
	where $R$ is the Riemann curvature tensor of $\omega$. Recall that if $(z_1, z_2)$ are holomorphic coordinates around $x$ in which $\omega = \sum_{i, j =1}^2 g_{i\overline{j}} idz_i d\overline{z_j}$  and $v= v_1 \partial / \partial z_1 +  v_2 \partial / \partial z_1$, $ w=  w_1 \partial / \partial z_1 +  w_2 \partial / \partial z_2$ then 
	$$ \mbox{Bisec}_{\omega} (v, w) = \sum_{i, j, k, l =1}^2 R_{i\overline{j}k\overline{l}} v_i \overline{v_j} w_k \overline{w_l} , $$
	where
	$$ R_{i\overline{j}k\overline{l}} = -g_{i\overline{j}, k\overline{l}} + \sum_{s, t =1}^2 g^{s\overline{t}} g_{i\overline{t}, k} g_{s\overline{j}, \overline{l}} . $$
	Indexes after the comma indicate differentiation and $(g^{i\overline{j}})$ denotes the inverse transpose of the positive Hermitian matrix $(g_{i\overline{j}})$, the index $i$ being for the rows and $j$ for the columns.
	
	In Appendix A of \cite{JMR} it is shown that if $\eta$ is a smooth K\"ahler form in the unit ball $B_1 \subset \mathbb{C}^2$, say, and $F$ is a smooth positive function such that
	\begin{equation} \label{loc rep}
	\omega= \eta + i \partial \overline{\partial} (F |z_1|^{2\beta})
	\end{equation}
	is positive on $B_1 \setminus \lbrace z_1 =0 \rbrace$; then there is $C>0$ such that $\mbox{Bisec}(\omega) \leq C$ on $B_{1/2} \setminus \lbrace z_1 =0 \rbrace$, say. As noticed in \cite{Rubinstein} this fact can also be derived by considering the pull-back of the smooth K\"ahler metric $\eta + i \partial \overline{\partial} (F|z_{n+1}|^2)$ in \(\mathbb{C}^{n+1}\) by the map \((z_1, \ldots, z_n) \to (z_1, \ldots, z_n, z_n^{\beta} ) \) and appealing to the well-known fact that the holomorphic sectional curvature of a K\"ahler sub-manifold is bounded above by that of the ambient space -see \cite{GriffithsHarris}-. At points $x \in C$ where $\chi_L (x) = 0$ the metric $\omega$ in Lemma \ref{AM} can't be written in the form \ref{loc rep}.
	
	We choose $\Lambda >0$ in \ref{ref} such that  $\omega_{B}$ can be written in the form \ref{loc rep} around the points of the curve, so we have an upper bound on $ \mbox{Bisec} (\omega_{B})$ on compacts sets. In order to extend this bound to $\mathbb{C}^2$ we use the `asymptotically conical' behavior of $\omega_{B}$.
	To prove  Lemma \ref{UBD}  it suffices to bound from above $\mbox{Bisec}(\omega_F + G^{*} \nu)$ in a region $U_{\delta_0, R_0}$ for some $\delta_0, R_0 >0$. Note that outside a compact set $G^{*} \nu = i\partial\overline{\partial}|p|^{\delta}$. 
	Let $0 \not= q \in L$ and $B$ a neighborhood of $q$ where there exist coordinates $(\xi_1, \xi_2)$ which  map $B$ to the unit ball in $\mathbb{C}^2$ in which $ \omega_F = |\xi_1|^{2\beta-2} i d\xi_1 d\overline{\xi}_1 +  i d\xi_2 d\overline{\xi}_2$. We might also assume that $ |q| \geq 2$ and that $B$ is contained in the euclidean ball of radius half the euclidean distance from $q$ to $0$.
	Let $m_{\lambda} : B \to \lambda B$ for $\lambda \geq 1$ be the multiplication by $\lambda$ in $\mathbb{C}^2$. We simplify notation and write $\nu$ for $G^{*}\nu$.  Then 
	$$m_{\lambda}^{*} (\omega_F + \nu )= \lambda^{2c} (\omega_F + \lambda^{\delta-2c}  \nu) . $$
	We will show that we have an upper bound for the bisectional curvature of  $\omega_F + \lambda^{-2c} m_{\lambda}^{*} \nu$ on $B_{1/2}$ which is independent of $\lambda \geq 1$. By a covering argument this gives the desired bound on $U_{\delta_0, R_0}$ and hence proves Lemma \ref{UBD}.
	
	Write  $ \nu_{i\overline{j}} = \nu (\frac{\partial}{\partial \xi_i}, \frac{\partial}{\partial \overline{\xi_j}})$. Let  $Q>0$ be such that 
	\begin{equation} \label{CB1}
	Q^{-1}(\delta_{ij}) \leq (\nu_{i\overline{j}} ) \leq Q (\delta_{ij}) ; \hspace{3mm}  | \nu_{i\overline{j}, k} | \leq Q; \hspace{3mm} | \nu_{i\overline{j}, k \overline{l}}| \leq Q 
	\end{equation}
	on $B_{1/2}$.
	Write  $\omega = \omega_{(\beta)} + \epsilon  \nu$  with   $\nu$ a smooth K\"ahler form in the unit ball in $\mathbb{C}^2$ ,  $0< \epsilon < 1$,
	
	$$ \omega_{(\beta)} = |\xi_1|^{2\beta-2} i d\xi_1 d\overline{\xi}_1 +  i d\xi_2 d\overline{\xi}_2$$ 
	and
	$$ \nu = \sum_{i,j =1}^2  \nu_{i\overline{j}} id\xi_i d\overline{\xi}_j . $$
	The desired bound then follows from the following
	
	\begin{lemma}
		There is a constant $C$, independent of $\epsilon >0$,  such that $\mbox{Bisec}(\omega) \leq C$ on $B_{1/2}$. In fact $C$ depends only on $Q$, where $Q>0$ is such  that,  on $B_{1/2}$ ,    $Q^{-1} \omega_{euc}  \leq \nu \leq Q \omega_{euc}$ and $|\nu_{i\overline{j}, k}|, |\nu_{i\overline{j}, k\overline{l}}|  \leq Q$ for any $i, j, k ,l$.
	\end{lemma}
	
	\begin{proof}
		This follows the lines of Appendix A in \cite{JMR}.
		Write
		$$ \omega = |\xi_1|^{2\beta -2} i d\xi_1 d\overline{\xi}_1 + \sum_{i, j =1}^2  \tilde{g}_{i\overline{j}} id\xi_i d\overline{\xi}_j , $$
		so that 
		
		$$\tilde{g}_{1\overline{1}} = \epsilon \nu_{1\overline{1}}, \hspace{3mm} \tilde{g}_{1\overline{2}} = \epsilon \nu_{1\overline{2}} \hspace{3mm} \tilde{g}_{2\overline{2}} = 1 + \epsilon \nu_{2\overline{2}} . $$
		Let $x=(x_1, x_2) \in B_{1/2} \setminus \lbrace \xi_1 =0 \rbrace$. Define new coordinates $(z_1, z_2)$ around $x$ via
		
		$$ \xi_1 = z_1$$
		$$ \xi_2 = z_2 + \frac{a}{2} (z_1 -x_1)^2 + b (z_1 - x_1)(z_2-x_2) + \frac{c}{2} (z_2 - x_2)^2 , $$
		where
		$$ a = - (\tilde{g}_{2\overline{2}} (x))^{-1} \tilde{g}_{1\overline{2}, 1} (x), \hspace{3mm} b = - (\tilde{g}_{2\overline{2}} (x))^{-1} \tilde{g}_{1\overline{2}, 2} (x), \hspace{3mm} c = - (\tilde{g}_{2\overline{2}} (x))^{-1} \tilde{g}_{2\overline{2}, 2} (x) .   $$
		In this new coordinates we have
		
		$$\omega = |z_1|^{2\beta-2} idz_1d\overline{z}_1 + \sum_{i,j} \hat{g}_{i\overline{j}} i dz_i d\overline{z}_j . $$
		
		\begin{claim}
			$ \hat{g}_{i\overline{j}, k} (x) = 0$ when $j \not= 1$. 
		\end{claim}
		
		Indeed, write $ d\xi_2 = A dz_1 + B dz_2$, with $A = a(z_1 -x_1) + b (z_2 -x_2)$ and $B= 1 + b(z_1 - x_1) + c(z_2 -x_2)$. A straightforward computation gives
		
		$$ \hat{g}_{1\overline{2}} = \tilde{g}_{1\overline{2}} \overline{B} + \tilde{g}_{2\overline{2}} A \overline{B}, \hspace{3mm} \hat{g}_{2\overline{2}} = |B|^2 \tilde{g}_{2 \overline{2}} . $$
		From here we get
		
		$$ \hat{g}_{1\overline{2}, 1}(x) = \tilde{g}_{1\overline{2}, 1} (x) + \tilde{g}_{2\overline{2}}(x) a, \hspace{3mm}  \hat{g}_{1\overline{2}, 2}(x) = \tilde{g}_{1\overline{2}, 2} (x) + \tilde{g}_{2\overline{2}}(x) b, \hspace{3mm} \hat{g}_{2\overline{2}, 2}(x) = \tilde{g}_{2\overline{2}, 2} (x) + \tilde{g}_{2\overline{2}}(x) c  . $$
		Our choice of $a, b, c$ implies that these three numbers are zero. The K\"ahler condition  $\hat{g}_{i\overline{j}, k} =  \hat{g}_{k\overline{j}, i}$ implies that $\hat{g}_{2\overline{2}, 1} (x) = \hat{g}_{1\overline{2}, 2} (x) = 0$ and the claim follows.
		
		We compute the bisectional curvature of $\omega$ at $x$ using the coordinates $(z_1, z_2)$.  Let  $v= v_1 \partial / \partial z_1 +  v_2 \partial / \partial z_1$ and $ w =  w_1 \partial / \partial z_1 +  w_2 \partial / \partial z_2   \in T_x^{1,0} \mathbb{C}^2$ with $ |v|_{\omega}= |w|_{\omega} =1$. Note that this implies that $|v_1|, |w_1| \leq C |z_1|^{1-\beta}$ and $|v_2|, |w_2| \leq C$.  Write $\omega = \sum_{i, j =1}^2 g_{i\overline{j}} idz_i d\overline{z_j}$ . So that  $g_{i\overline{j}}= \hat {g}_{i\overline{j}}$ when $(i, j) \not= (1,1)$ and $ g_{1\overline{1}}= |z_1|^{2\beta-2} + \hat {g}_{1\overline{1}}$. Write $ \mbox{Bisec}_{\omega} (v, w) = T_1 + T_2$, where
		
		$$ T_1 = - \sum_{i, j, k, l} g_{i\overline{j}, k\overline{l}} (x) v_i \overline{v}_j w_k \overline{w}_l$$
		and
		$$ T_2 =  \sum_{s, t, i , j , k, l =1}^2 g^{s\overline{t}}(x) g_{i\overline{t}, k}(x) g_{s\overline{j}, \overline{l}}(x)  v_i \overline{v}_j w_k \overline{w}_l . $$
		
		\begin{claim}
			$$ T_1 \leq C -(\beta-1)^2 |z_1|^{2\beta-4} |v_1|^2 |w_1|^2 . $$
		\end{claim}
		In fact $g_{1\overline{1}, 1\overline{1}} = (\beta-1)^2 |z_1|^{2\beta-4} + \hat{g}_{1\overline{1}, 1\overline{1}}$, and we have
		$$ \hat{g}_{1\overline{1}} = \tilde{g}_{1\overline{1}}  + A \tilde{g}_{2\overline{1}} + \overline{A} \tilde{g}_{1\overline{2}} + |A|^2 \tilde{g}_{2\overline{2}}   .    $$
		From here we compute
		$$ \hat{g}_{1\overline{1}, 1\overline{1}} (x) = \tilde{g}_{1\overline{1}, 1\overline{1}} (x) + a \tilde{g}_{2\overline{1}, \overline{1}} (x) + \overline{a} \tilde{g}_{1\overline{2}, 1} (x) + |a|^2 \tilde{g}_{2\overline{2}} . $$
		Since the differential at $x$ of the change of coordinates between $(\xi_1, \xi_2)$ and $(z_1, z_2)$ is the identity, we have that
		$$ \tilde{g}_{i\overline{j}, k} (x)= \frac{\partial \tilde{g}_{i\overline{j}}}{\partial \xi_k} (x)= \frac{\partial \nu_{i\overline{j}}}{\partial \xi_k} (x), \hspace{3mm}  \tilde{g}_{i\overline{j}, k\overline{l}} (x)= \frac{\partial^2 \tilde{g}_{i\overline{j}}}{\partial \overline{\xi}_l \partial \xi_k} (x)= \frac{\partial^2 \nu_{i\overline{j}}}{\partial \overline{\xi}_l\partial \xi_k} (x) . $$
		From this fact, and $|a| = |- (\tilde{g}_{2\overline{2}} (x))^{-1} \tilde{g}_{1\overline{2}, 1} (x)| \leq |\tilde{g}_{1\overline{2}, 1} (x)|$ we get that $|\hat{g}_{1\overline{1}, 1\overline{1}} (x)| \leq C$. Similarly, when $(i, j, k, l) \not= (1,1,1,1)$ we have $ |g_{i\overline{j}, k\overline{l}}(x)| = |\hat{g}_{i\overline{j}, k\overline{l}} (x)| \leq C$, and the claim follows.
		
		\begin{claim}
			$$ T_2 \leq C + (\beta-1)^2 |z_1|^{2\beta-4} |v_1|^2 |w_1|^2 . $$
		\end{claim}
		Define a non-negative bilinear Hermitian form on tensors $ a=[a_{i\overline{j}k}]$ satisfying $a_{i\overline{j}k}= a_{k\overline{j}i}$ by
		
		$$ \langle [ a_{i\overline{j}k}], [b_{p\overline{q}r}] \rangle = \sum g^{q\overline{j}}(x) (w_i a_{i\overline{j}k} v_k) \overline{(w_p b_{p\overline{q}r})v_r)} . $$
		Then
		
		$$T_2= \| D + E \|^2$$
		with $D_{ijk} = \hat{g}_{i\overline{j}, k}$ and $E_{ijk} = (\beta-1) |z_1|^{2\beta-4} \overline{z_1}$ if $(ijk) = (111)$ and $E_{ijk} = 0 $ otherwise. We first estimate
		
		$$ \| E \|^2 = (\beta-1)^2 |z_1|^{4\beta-6} g^{1\overline{1}}(x) |v_1|^2 |w_1|^2 , $$
		where $g^{1\overline{1}} = \det(g)^{-1} g_{2\overline{2}}$.
		
		$$ \det(g) = (|z_1|^{2\beta -2} + \hat{g}_{1\overline{1}}) \hat{g}_{2\overline{2}} - |\hat{g}_{1\overline{2}}|^2 = \hat{g}_{2\overline{2}} |z_1|^{2\beta-2} \left( 1 + (\hat{g}_{2\overline{2}})^{-1} \det(\hat{g}) |z_1|^{2-2\beta}\right) . $$
		Unwinding notation we have that at the point $x$, $ \hat{g}_{2\overline{2}} = 1 + \epsilon \nu_{2\overline{2}}(x)$ and $\det(\hat{g}) (x) = \epsilon \nu_{1\overline{1}}(x) + \epsilon^2 \det (\nu) (x)$. We conclude that $(\hat{g}_{2\overline{2}})^{-1} \det(\hat{g}) \geq Q^{-1} \epsilon $, so 
		
		$$ g^{1\overline{1}}(x) \leq ( 1 + \delta)^{-1} |z_1|^{2-2\beta}$$
		with $\delta = Q^{-1} \epsilon |z_1|^{2-2\beta}$. We get
		
		$$ \| E \|^2 \leq (1 + \delta)^{-1} (\beta-1)^2 |z_1|^{2\beta-4} |v_1|^2 |w_1|^2 . $$
		Next we do a trick
		
		$$ \| T_2 \|^2 \leq ( 1 + \delta^{-1} ) \| D \|^2 + (1+ \delta) \| E \|^2 . $$
		The claim (and the lemma) will follow if we can bound $ \epsilon^{-1} |z_1|^{2\beta-2} \| D \|^2$.
		
		$$ \| D \|^2 =  \sum_{s, t, i , j , k, l =1}^2 g^{s\overline{t}}(x) \hat{g}_{i\overline{t}, k}(x) \hat{g}_{s\overline{j}, \overline{l}}(x)  v_i \overline{v}_j w_k \overline{w}_l =   \sum_{ i , j , k, l =1}^2 g^{1\overline{1}}(x) \hat{g}_{i\overline{1}, k}(x) \hat{g}_{1\overline{j}, \overline{l}}(x)  v_i \overline{v}_j w_k \overline{w}_l  . $$
		(The second equality follows from the first claim.) Since $ g^{1\overline{1}} (x) \leq |z_1|^{2-2\beta} $ and $ | \hat{g}_{i\overline{j}, k} (x) | \leq C \epsilon$, the estimate follows.
		
	\end{proof}

	We summarize the results obtained so far into the following
	
	\begin{proposition} \label{RMETRICS}
		There exist $H \in \mathcal{D}$ and  K\"ahler metrics $\omega$, $\omega_{B}$ on $\mathbb{C}^2 $ with cone singularities of angle $2 \pi \beta$ along $C$ such that
		\begin{itemize}
			\item $|(H^{-1})^{*} \omega - \omega_F|_{g_F} = O( r^{-1/c})$ 
			\item   Bisec($\omega_{B}$) $\leq Q_1$
			\item $ Q_2^{-1}\omega_{B} \leq \omega \leq Q_2 \omega_{B}$
			\end {itemize}
			
			for some positive constants $Q_1, Q_2$.
			
		\end{proposition}
		
		We review the definition of a metric having cone singularities in Subsection \ref{mcs}. The statement about the singularities follows from the fact that around points of $C$ one can write the metrics as (smth) $+ i\partial \overline{\partial}(F|z_1|^{2\beta})$ with  $F$ a smooth positive function and where (smth) denotes a smooth  $(1,1)$ form positive in the direction tangent to $C$ -see Lemma \ref{CSing}-. 
		The metric $\omega$ is isometric to the flat metric $\omega_F$ in the neighborhood  of $C$ at infinity, $U_{\delta/2, 2R}$.

	\subsection{The Sobolev inequality} \label{SOBOLEV}
	Around points of the curve $C$ we can find complex coordinates \(z_1, z_2\) such that \(C=\{z_1=0\}\) and the reference metric \(\omega\) is quasi-isometric to the model \(g_{(\beta)}= \beta^2 |z_1|^{2\beta-2} |dz_1|^2 + |dz_2|^2.\) On the other hand if we write $z_1= r_1^{1/\beta} e^{i\theta_1}$ we have $g_{(\beta)} = dr_1^2 + \beta^2 r_1^2 d\theta_1^2 + |dz_2|^2$, in these coordinates \(g_{(\beta)}\) is quasi-isometric to the standard Euclidean metric. In other words there are two relevant differential structures in our context, one is given by the complex coordinates we started with, the other is given by declaring the `cone coordinates' \((r_1 e^{i\theta_1}, z_2)\) to be smooth. The two structures are clearly equivalent by a map modeled on $r_1 e^{i\theta_1} \to r_1^{1/\beta}e^{i\theta_1}$ in transverse directions to $C$.
	  
	Recall that the Sobolev inequality says that there exists a constant $C$ such that for every smooth function $\phi$ in $\mathbb{R}^4$ with compact support, we have
	\begin{equation} \label{SOBIN}
	\left( \int_{\mathbb{R}^4} |\phi|^4 \right)^{1/2} \leq C \int_{\mathbb{R}^4} | \nabla \phi |^2 .
	\end{equation}
	It is clear that \ref{SOBIN} holds if the Euclidean metric is replaced by a quasi-isometric metric $g$, i.e. $\Lambda^{-1}g \leq g_{euc} \leq \Lambda g$ for some constant $\Lambda>1$. Consider now the case of the flat metric $g_F$. We claim that there exists a diffeomorphism $\Phi$ of   $\mathbb{C}^2 \setminus L$ such that   $\Lambda^{-1}g_{euc} \leq \Phi^{*}g_F \leq \Lambda g_{euc}$ for some $\Lambda >1$. This is indeed clear from the construction of $g_F$ : The spherical metric with cone singularities on $\mathbb{CP}^1$ is, up to a diffeomorphism, quasi-isometric to the round metric and the same is true for the singular metric $\overline{g}$ on the three-sphere. We conclude that the Sobolev inequality \ref{SOBIN} holds for $g_F$. Consider now our reference metric $\omega$, given by Lemma \ref{AM}. We claim that there exists a diffeomorphism $\Psi$ of $\mathbb{C}^2 \setminus C$ such that $\Lambda^{-1}\omega_{euc} \leq \Psi^{*}\omega \leq \Lambda \omega_{euc}$ for some $\Lambda >0$. Indeed $\Phi \circ H$ will do the job outside a compact set. It is easy to patch this with a diffeomorphism supported in a tubular neighborhood of the curve, modeled on $r_1 e^{i\theta_1} \to r_1^{1/\beta}e^{i\theta_1}$ in transverse directions to $C$. The claim follows. As a result of the discussion we have the following
	\begin{proposition}
		There exists a constant $C$ such that the Sobolev inequality \ref{SOBIN} holds for the reference metric $\omega$ and all functions $\phi$ with compact support, smooth in the cone coordinates.
	\end{proposition}
	
	\section{Linear analysis} \label{la}
	
	We define Banach spaces of continuous functions on $\mathbb{C}^2$ on which the Laplacian of the reference metric $\omega$ acts as a Fredholm operator. The main result is Proposition \ref{LTHRM}. To set up the corresponding linear theory for ALE metrics, Joyce invokes the explicit expression of the Green's function of the Euclidean space. We avoid those arguments by means of some others conventional methods used in the study of AC manifolds. In a few words we can say that,  provided with Donaldson's interior Schauder estimates, the standard techniques used to establish the linear theory for AC manifolds work in our setting.

	\subsection{Interior Schauder estimates} \label{mcs}
	
	Consider the model metric $g_{(\beta)} = \beta^2 |z_1|^{2\beta-2} |dz_1|^2 + |dz_2|^2$ on $\mathbb{C}^2$. We want to define H\"older continuous $(1,0)$ and $(1,1)$ forms. Write $z_1= r_1^{1/\beta} e^{i\theta_1}$ so that $g_{(\beta)} = dr_1^2 + \beta^2 r_1^2 d\theta_1^2 + |dz_2|^2$. Set $\epsilon = dr_1 + i \beta r_1 d\theta_1$. A $(1,0)$ form $\eta$ is called $C^{\alpha}$  if $\eta= f_1 \epsilon + f_2 dw$ with $f_1, f_2$ $C^{\alpha}$ functions in the usual sense in the cone coordinates $ (r_1 e^{i\theta_1}, z_2)$. It is also required that $f_1  =0$  on $\lbrace z_1 = 0 \rbrace$. If we change $\epsilon$ by $\tilde{\epsilon} = e^{i\theta} \epsilon = \beta |z_1|^{\beta-1} dz_1$, say,  in the definition; then the vanishing condition implies that we get the same space. In order to define $C^{\alpha}$ $(1,1)$ forms we use the basis $\lbrace \epsilon \overline{\epsilon}, \epsilon d\overline{w}, dw\overline{\epsilon}, dwd\overline{w} \rbrace$,  as above we ask the components to be $C^{\alpha}$ functions and we require the components  corresponding to $\epsilon d\overline{w},  dw \overline{\epsilon}$ to vanish on the singular set.  Finally we set $C^{2, \alpha}$ to be the space of $C^{\alpha}$ functions $u$ such that $\partial u$ and $\partial \overline{\partial} u$ are $C^{\alpha}$. We define the $C^{\alpha}$ norm of a function $\|f\|_{\alpha}$ as the sum of its $C^0$ norm $\|f\|_0$ and its $C^{\alpha}$ semi-norm $[f]_{\alpha}$; in the cone coordinates this last semi-norm agrees with the standard
	$$ [f]_{\alpha} = \sup_{x, y} \frac{|f(x) - f(y)|}{|x-y|^{\alpha}} . $$
	To define the $C^{2, \alpha}$ norm of a function $f$ we simply add  $\|f\|_{\alpha}$,  the $C^{\alpha}$ norm of the components of $\partial f$ in the basis $ \{ \epsilon, dw \}$ and the $C^{\alpha}$ norm of the components of $i \partial \overline{\partial} f$ in the  basis  $\lbrace \epsilon \overline{\epsilon}, \epsilon d\overline{w}, dw\overline{\epsilon}, dwd\overline{w} \rbrace$.
	
	We are interested in the equation $\triangle u =f$, where $\triangle$ is the Laplace operator of $g_{(\beta)}$. We define $L^2_1$ on domains of $\mathbb{C}^2$  by means of the usual norm $ \| u \|_{L^2_1} = \int |\nabla u|^2 + \int u^2$.  In the cone coordinates $(r_1 e^{i\theta_1}, z_2)$,  $ \beta^2 g_{euc} \leq g_{(\beta)} \leq (1+\beta^2) g_{euc}$ and  $L^2_1$ coincides with the standard Sobolev space. Let $u$ be a function that is locally in $L^2_1$. We say that $u$ is a weak solution of $\triangle u =f$ if
	$$ \int \langle\nabla u, \nabla \phi \rangle = - \int f\phi$$ 
	for all smooth compactly supported $\phi$.

	Fix $\alpha < \beta^{-1} -1$ and let $u$ be a weak solution of $\triangle u = f$ on $B_2$ with $f \in C^{\alpha}(B_2)$. Then \cite{Donaldson} shows that $u \in C^{2, \alpha}(B_1)$ and there is a constant $C$ -independent of $u$- such that 
		\begin{equation} \label{SCH}
		\|u\|_{C^{2, \alpha}(B_1)} \leq C \left( \|f\|_{C^{ \alpha}(B_2)} + \|u\|_{C^0(B_2)} \right) .
		\end{equation}

	We mention 3 differences between this result and the standard Schauder estimates
	\begin{itemize}
		\item We don't have estimates for all the second derivatives of $u$. (E.g. $\partial^2 u / \partial r_1^2$).
		\item If $\triangle u \in C^{\alpha}$ then the component of $\partial u$ corresponding to $\epsilon$ needs to vanish along the singular set.
		\item The estimates require $\alpha < \beta^{-1} -1$.
	\end{itemize}
	
	These differences  can be explained by the fact that if $p$ is a point outside the singular set  and $\Gamma_p = G(., p)$, where $G$ is the Green's function for $\triangle$;  then around points of $\lbrace z_1 =0 \rbrace$ one can write a convergent series expansion
	
	\begin{equation} \label{expansion}
	\Gamma_p = \sum_{j, k \geq 0} a_{j, k} (z_2) r_1^{(k/\beta) + 2j} \cos (k\theta_1)
	\end{equation}
	with $a_{j, k}$ smooth functions.
	The proof of the Schauder estimates in \cite{Donaldson} uses classical methods. The expression \ref{expansion} is proved by separation of variables and a check of convergence. The coefficients $a_{j, k}$ are given in terms of Bessel's functions. 
	If $u$ is a function with compact support such that $\triangle u = f$, then
	\begin{equation} \label{green}
	u(x) = \int G(x, y) f(y) dy .
	\end{equation}
	To show the estimate \ref{SCH} one has to differentiate \ref{green} twice. The proof then follows the one of the standard  Scahuder estimates with some modifications due to the fact that $\triangle$ is not translation invariant.

	Let $\eta$ be a  $(1, 1)$ form on $B_2$ with $\| \eta \|_{C^{\alpha}(B_2)} \leq \epsilon$. Assume that $\eta$ has support contained in $B_1$ and consider the operator $ Lu = \triangle u + \langle \partial \overline{\partial} u, \eta \rangle$. If $\epsilon <1/(2C)$ we can use \ref{SCH} to get the estimate
	\begin{equation} \label{prev}
	\|u\|_{C^{2, \alpha}(B_1)} \leq 2C \left( \|Lu\|_{C^{ \alpha}(B_2)} + \|u\|_{C^0(B_2)} \right)
	\end{equation}
	for all functions  $u \in C^{2, \alpha}(B_2)$.
	
	Now let $C$ be our smooth curve in $\mathbb{C}^2$ and let $\omega$ be a (smooth) K\"ahler metric in the complement of $C$. We say that $\omega$ is a metric with cone singularities along $C$ of angle $2\pi \beta$ if around each $p \in C$ we can find holomorphic coordinates $(z_1, z_2)$ such that 
	\begin{equation} \label{DEF}
	\omega = \omega_{(\beta)} + \eta
	\end{equation}
	with $\eta \in C^{\alpha}$ and $\eta(p) =0$. More precisely,  $\eta(p) =0$  means that the coefficients of $\eta$ in the basis  $\lbrace \epsilon \overline{\epsilon}, \epsilon d\overline{w}, dw\overline{\epsilon}, dwd\overline{w} \rbrace$ vanish at $p$. 
	
	Given our curve $C$ and a bounded open subset $U$ of $\mathbb{C}^2$ we can define the space $C^{2, \alpha}(U)$ by taking a finite cover of $U$ with coordinates in which $C= \lbrace z_1 =0 \rbrace$. Let $p \in C$ and write $ \omega$ as in \ref{DEF}. After a dilation and multiplying by a cut-off function we can assume that in a smaller neighborhood of $p$ we have $\triangle_{\omega} = L$ with $L$ as in \ref{prev}. From here we get that
	
	\begin{equation} \label{LAP}
	\|u\|_{C^{2, \alpha}(U)} \leq C \left( \| \triangle_{\omega} u \|_{C^{ \alpha}(V)} + \|u\|_{C^0(V)} \right)
	\end{equation}
	for all $u \in C^{2, \alpha}(V)$. In \ref{LAP} we assume that $U$ is compactly contained in $V$. The constant $C$ depends on $\omega, U, V$.

	Finally we note that our reference metrics in Proposition \ref{RMETRICS} have cone singularities as we have defined because of the following
	
	\begin{lemma} \label{CSing}
		Let $\omega$ be a K\"ahler metric on $\mathbb{C}^2 \setminus C$ such that around each $p \in C$ we can find holomorphic coordinates $(z_1, z_2)$ where
		$$ \omega = \Omega + i \partial \overline{\partial} (F|z_1|^{2\beta})$$
		with $\Omega$ a smooth $(1, 1)$ form such that $\Omega(\partial/\partial z_2, \partial/\partial \overline{z}_2) (p) >0$ and $F$ a smooth positive function, then $\omega$ has cone singularities in the sense of \ref{DEF}.
	\end{lemma}
	
	\begin{proof}
		This follows from the computation
		
		$$ i \partial \overline{\partial} (F |z_1|^{2\beta}) = |z_1|^{2\beta} i \partial \overline{\partial} F + \beta |z_1|^{2\beta-2} \left( \overline{z}_1 idz_1 \overline{\partial}F + z_1 \partial F d\overline{z}_1 \right) + \beta^2 F |z_1|^{2\beta -2} idz_1 d\overline{z}_1 . $$
		Set $\tilde{z_1} = a z_1$, $\tilde{z_2} = b z_2$ with $a = F(p)^{1/2}$ and $b=\left( \Omega(\partial/\partial z_2, \partial/\partial \overline{z}_2) (p) \right)^{1/2}$ to get \ref{DEF}
		
	\end{proof}
	
	\subsection{Weighted H\"older spaces} \label{wh}
	
	We introduce weights to the previous H\"older spaces. In this subsection we work  with the flat metrics $g_F$ from Section \ref{Fmetrics}. The property we shall exploit the most is the one of being a metric cone. If $\gamma$ is the weight parameter in our space of functions, then a function in the space is bounded by $r^{\gamma}$. In particular, if $\gamma < 0$, we allow our functions to blow up at the apex of the cone. 
	
	Let $g_F$ be the flat metric. Write $B_R = \lbrace r < R \rbrace$ for the metric ball of radius $R$ around the origin.  Consider the annulus $A_1 = B_2 \setminus \overline{B_1}$ and the  bigger one $ \tilde{A_1} = B_4 \setminus \overline{B_{1/2}}$. We know that around each $p \in L \cap A_1$ we can find coordinates $(z_1, z_2)$ in which $g_F = g_{(\beta)}$ and that $g_F$ is locally isometric to the euclidean metric outside $L$. We fix a finite cover of $A_1$ by such coordinates and define the spaces $C^{\alpha} (A_1)$ and $C^{2, \alpha}(A_1)$ in the obvious way. Alternatively (in more intrinsic terms)  we can define the space $C^{\alpha}$ functions in any domain by considering the distance induced by $g_F$ and applying the standard definition. To measure the $C^{2, \alpha}$ norm of a function   we can take an orthonormal basis for the $(1,0)$ forms $\lbrace \tau_1, \tau_2 \rbrace$, for example by applying Gram-Schmidt to $\lbrace dz, dw \rbrace$ over $A_1 \setminus L$,  and sum the $C^{\alpha}$ norm of  the components of $\partial u$ and $\partial \overline{\partial} u$ with respect to $\tau_i$ and $\tau_i \overline{\tau_j}$ respectively. The result is independent of the choice of orthonormal basis $\{ \tau_1, \tau_2 \}$.  One can replace $A_1$ with $\tilde{A_1}$  in the above discussion without any change.
	It follows from the interior Schauder estimates that there is a constant $C$ such that for every $u \in C^{2, \alpha}(\tilde{A_1})$
	
	\begin{equation} \label{int estimate}
	\|u\|_{C^{2, \alpha}(A_1)} \leq C \left( \|f\|_{C^{ \alpha}(\tilde{A_1})} + \|u\|_{C^0(\tilde{A_1})} \right) , 
	\end{equation}
	where $\triangle u = f$ is the Laplacian  of $u$ with respect to $g_F$.

	Let $\gamma \in \mathbb{R}$, we want to define the space $C^{\alpha}_{\gamma}$.  For $\lambda > 0$,  denote $A_{\lambda} = B_{2\lambda} \setminus B_{\lambda}$. In other words $A_{\lambda} = D_{\lambda}(A_1)$ where $D_{\lambda}$ is the map given in spherical coordinates by $D_{\lambda} (r, \theta) = (\lambda r , \theta)$. Note that in complex coordinates $D_{\lambda} (z, w) = (\lambda^{1/c} z, \lambda^{1/c}w)$. Let $f$ be a continuous function on $\mathbb{C}^2 \setminus \lbrace 0 \rbrace$. Define $ f_{\lambda, \gamma} = \lambda^{-\gamma} . (f \circ D_{\lambda})$ and  think of it as a function on $A_1$. Finally we set
	
	\begin{equation} \label{norm1}
	\| f \|_{\alpha, \gamma} = \sup_{\lambda > 0} \|f_{\lambda, \gamma} \|_{C^{\alpha}(A_1)} .
	\end{equation}
	It follows that if $f \in C^{\alpha}_{\gamma}$ (the space of functions in $\mathbb{C}^2 \setminus \lbrace 0 \rbrace$  for which the above norm is finite),  then $|f(x)| \leq A r(x)^{\gamma}$ for some constant $A$. In fact, if we let $\| f \|_{0, \gamma} = \sup_{\lambda > 0} \|f_{\lambda, \gamma} \|_{C^0(A_1)}$ we clearly have $\|f\|_{0, \gamma} \leq \|f \|_{\alpha, \gamma}$ and $ \| f\|_{0, \gamma}$ is easily seen to be equivalent to $ \sup_{x} r(x)^{-\gamma} |f(x)|$.
	It is clear that if we use $\tilde{A_1}$ instead we would get an equivalent norm, i.e, there exist a constant $C$ such that
	
	$$ \sup_{\lambda > 0} \|f_{\lambda, \gamma} \|_{C^{\alpha}(\tilde{A_1})} \leq C  \| f \|_{\alpha, \gamma} . $$
	Having said what is the space $C^{2, \alpha}$ on $A_1$ we can define the space $C^{2, \alpha}_{\delta}$ to be the space of functions $u$ on $\mathbb{C}^2 \setminus \lbrace 0 \rbrace$  for which 
	
	\begin{equation} \label{norm2}
	\| u \|_{2, \alpha, \delta} = \sup_{\lambda > 0} \|u_{\lambda, \delta} \|_{C^{2, \alpha}(A_1)}
	\end{equation}
	is finite. As above $\delta$ is any fixed real number. 
	
	With these definitions we claim that $\triangle$ defines a bounded operator from $C^{2, \alpha}_{\delta}$ to $C^{\alpha}_{\delta -2}$. Indeed,  from the expression
	
	\begin{equation} \label{lap}
	\triangle = \frac{\partial^2}{\partial r^2} + \frac{3}{r} \frac{\partial}{\partial r} + \frac{1}{r^2} \triangle_{\overline{g}} , 
	\end{equation}
	we get that $\triangle u_{\lambda} = \lambda^2 (\triangle u)_{\lambda}$.  We denote $u_{\lambda} = u \circ D_{\lambda}$. Now take  $ u \in C^{2, \alpha}_{\delta}$,  write $\triangle u =f$ and let $\lambda > 0$. Then
	
	$$ f_{\lambda, \delta -2} =  \lambda^{-\delta + 2} (\triangle u)_{\lambda}= \lambda^{-\delta} \triangle u_{\lambda}$$
	and our claim follows from the fact that $\triangle : C^{2, \alpha}(A_1) \to C^{\alpha}(A_1)$ is a bounded operator. 
	
	Let us give an equivalent norm in $C^{2, \alpha}_{\delta}$ which will make evident the fact that if $u$ belongs to this space then $| \partial \overline{\partial} u |_{g_F} = O( r^{\delta -2})$. In order to do this we note that on $\mathbb{C}^2 \setminus L$ we have an (up to a factor of $\sqrt{2}$) orthonormal basis $\lbrace \tau_1 , \tau_2 \rbrace$ (w.r.t. $g_{F}$) of the $(1, 0)$ forms such that $ D_{\lambda}^{*} \tau_i = \lambda \tau_i $. Given a function $u$ we write $\partial u = \sum_i u_i \tau_i$ and $\partial \overline{\partial} u = \sum_{i, j} u_{i\overline{j}} \tau_i \overline{\tau_j}$. We claim that
	\begin{equation} \label{norm3}
	\|u\|_{2, \alpha, \delta} = \|u\|_{0, \delta} + \sum_i \|u_i\|_{\alpha, \delta -1} + \sum_{i,j} \|u_{i\overline{j}}\|_{\alpha, \delta-2}
	\end{equation}
	defines an equivalent norm as the previous one. (Our claim justifies the abuse of notation since \ref{norm3} is not exactly equal to \ref{norm2}.) Since $\triangle u = u_{1\overline{1}} + u _{2\overline{2}}$ we see again that $\triangle: C^{2, \alpha}_{\delta} \to C^{\alpha}_{\delta -2}$ is a bounded map.
	We compute $\|u_{\lambda, \delta} \|_{C^{2, \alpha}(A_1)}$ using the basis $\lbrace \tau_1, \tau_2 \rbrace$. Since $D_{\lambda}$ is holomorphic we have that $\partial u_{\lambda} = D_{\lambda}^{*} \partial u = \lambda \sum_i (u_i)_{\lambda} \tau_i$ and that $\partial \overline{\partial} u_{\lambda} = D_{\lambda}^{*} \partial \overline{\partial} u = \lambda^2 \sum_{i, j} (u_{i\overline{j}})_{\lambda} \tau_i \overline{\tau_j}  $. Our claim then follows from
	$$\|u_{\lambda, \delta} \|_{C^{2, \alpha}(A_1)}= \| \lambda^{-\delta} u_{\lambda} \|_{C^0(A_1)} + \sum_i \| \lambda^{-\delta +1} (u_i)_{\lambda} \|_{C^{\alpha}(A_1)} + \sum_{i,j} \| \lambda^{-\delta +2} (u_{i\overline{j}})_{\lambda} \|_{C^{\alpha}(A_1)} . $$
	In arguments in which the H\"older exponent $\alpha$ is not crucially needed we will say that a function is in $C^2$ if the components $u_{i\overline{j}}$ are continuous. Similarly we can give a definition of $C^2_{\delta}$.

	We are now ready  to state our first main estimate
	
	\begin{lemma} \label{Est 1}
		Let $\alpha < \beta^{-1} -1$ and $ \delta \in \mathbb{R}$. Then there is a constant $C = C(\alpha, \delta)$ such that for every $u \in C^{2, \alpha}_{\delta}$ with $\triangle u =f$ 
		$$ \|u\|_{2, \alpha, \delta} \leq C \left( \|f\|_{\alpha, \delta - 2} + \| u \|_{0, \delta} \right) . $$
	\end{lemma}
	
	\begin{proof}
		Write $\delta = \gamma +2$. Let $\lambda > 0$ we apply the interior estimate \ref{int estimate} to $u_{\lambda, \delta} = \lambda^{-\delta} u_{\lambda}$ to get
		
		$$ \|u_{\lambda, \delta} \|_{C^{2, \alpha}(A_1)} \leq C \left( \| \lambda^{-\delta +2} f_{\lambda} \|_{C^{\alpha}(\tilde{A_1})} + \| \lambda^{-\delta} u_{\lambda} \|_{C^0(\tilde{A_1})} \right) .$$
		Note that the first term on the r.h.s. is bounded by $\|f\|_{\alpha, \gamma}$ and the second term is bounded by $\|u\|_{0, \gamma +2}$.
		
	\end{proof}
	
	\begin{remark}
		In fact we have proved that if $u$ is locally in $C^{2, \alpha}$, $\triangle u \in C^{\alpha}_{\delta -2}$ and $\| u \|_{0, \delta}$ is finite, then $u \in C^{2, \alpha}_{\delta}$ and the above estimate holds.
	\end{remark}
	
	Our next goal is to bound $\|u\|_{0, \delta}$ in terms of $\|f\|_{\alpha, \delta - 2}$. It turns out that this is true,  except when $\delta$ belongs to the discrete set of `Indicial Roots'. In order to explain what is this set we digress a little and discuss some basics of spectral theory for $\triangle_{\overline{g}}$, the Laplacian of the singular metric on the 3-sphere. We invoke the compact embedding of $L^2_1$ in $L^2$  and the spectral theorem, our arguments are quite standard -see Section 0.6 in \cite{GriffithsHarris}-.  
	
	First we note that on $(S^3, \overline{g})$ there is an obvious definition of the spaces $L^2$ and $L^2_1$. Since there is a diffeomorphism $\chi$ of $S^3 \setminus L$ such that $\chi^{*} \overline{g}$ is quasi-isometric to a smooth metric on $S^3$ we see that $L^2$ and $L^2_1$ correspond under $\chi$ to the usual spaces. In particular we have that $L^2_1 \subset L^2$ is compact. 
	If we write the norms as $ \|f\|_{L^2}^2 = \int f^2$ and $\| u \|_{L^2_1}^2 = \int u^2 + \int | \nabla u |^2 $ we see that $ f \in L^2$ defines a bounded linear functional $T$ on $L^2_1$ by $ T(\phi) = \int f\phi$. If $u$ is such that $ T = \langle u, - \rangle_{L_1^2}$ then $u$ is said to be a weak solution of $-\triangle_{\overline{g}} u + u = f$. The map $K(f) = u$ is a bounded linear map between $L^2$ and $L^2_1$, composing this map with the compact inclusion we have a map $K: L^2 \to L^2$ which is compact and self-adjoint. It follows from the spectral theorem that we can find an orthonormal basis $\lbrace \phi_i \rbrace_{i\geq 0}$ of $L^2$ such that $K(\phi_i) = s_i \phi_i$ and $s_i \to 0$. Unwinding the definitions we get that $ \triangle_{\overline{g}} \phi_i = - \lambda_i \phi_i$ with $0= \lambda_0 \leq \lambda_1 \leq \lambda_2 \leq \ldots $ and $\lambda_i = (1-s_i)/s_i \to \infty$.
	For each $\lambda_i$ define $ \delta_i^{\pm}$ to be the solutions of $ \lbrace s(s+2) = \lambda_i \rbrace$ with $\delta_i^{+}$ non-negative  and $\delta_{i}^{-}$ non-positive (in fact $\leq -2$). The set of Indicial Roots is set to be $ I = \lbrace \delta_i^{\pm} , i \geq 0 \rbrace$. With this definition we can state the following
	
	\begin{lemma} \label{inj}
		Let $ u \in C^2_{\delta} $ be such that $\triangle u =0$ and  $ \delta \notin I$. Then $u=0$. 
	\end{lemma}
	
	\begin{proof}
		Write $ u (r, \theta) = \sum_{i=0}^{\infty} u_i (r) \phi_i (\theta)$, where $ u_i (r) = \int_{S^3} u(r, .) \phi_i$. It follows from H\"older's inequality that if $ |u| \leq C r^{\delta}$ then $ |u_i (r) | \leq C (\mbox{Vol} (\overline{g}))^{1/2} r^{\delta}$. On the other hand the equation $\triangle u =0$ implies 
		$$ u_i'' + \frac{3}{r} u_i' - \frac{\lambda_i} { r^2} u_i =0 , $$
		so that $u_i = A r^{\delta_i^{+}} + B  r^{\delta_i^{-}}$  for some constants $A$ and $B$. Since $\delta \not= \delta_i^{\pm}$ we get that $u_i=0$.
		
	\end{proof}
	
	\begin{proposition}
		Let $\alpha < \beta^{-1} -1$ and $ \delta \in \mathbb{R} \setminus I$. Then there is $C=C(\alpha, \delta)$ such that
		\begin{equation} \label{Est 2}
		\|u\|_{2, \alpha, \delta} \leq C \|f\|_{\alpha, \delta -2}
		\end{equation}
		for every $u \in C^{2, \alpha}_{\delta}$ with $ \triangle u = f$ . 
	\end{proposition}
	
	\begin{proof}
		If the result was not true then we would be able to take a sequence $\{ u_k \}$ with $ \| u_k \|_{2, \alpha, \delta} = 1$ , $ \triangle u_k = f_k$ and $\| f_k \|_{\alpha, \delta -2} \to 0$. It follows from Lemma \ref{Est 1} that $\|u_k\|_{0, \delta} \geq 2 \epsilon$ for some $\epsilon > 0$. Hence we can find $x_k$ such that $ r(x_k)^{-\delta} | u_k (x_k) | \geq \epsilon$. Consider the sequence $\tilde{u_k} = (u_k)_{L_k, \delta}$ where $L_k = r(x_k)$. Write $ x_k = (r(x_k), \theta_k)$,  then $| \tilde{u_k} (\tilde{x_k}) | \geq \epsilon$ with $\tilde{x_k} = (1, \theta_k)$. On the other hand $ \tilde{f_k} = \triangle \tilde{u_k} =  L_k^{-\delta + 2} (f_k)_{L_k} = (f_k)_{L_k, \gamma}$, with $\gamma = \delta -2$. The key point is that $ \| u \|_{2, \alpha, \delta} = \| u_{L, \delta} \|_{2, \alpha, \delta}$ and $ \| f \|_{\alpha, \gamma} = \| f_{L, \gamma} \|_{\alpha, \gamma}$ for any $L >0$ and $f,g$ any functions. So that $ \| \tilde{u_k} \|_{2, \alpha, \delta} = 1$ and $\| \tilde{f_k} \|_{\alpha, \delta -2} \to 0$. Let $K_n = \overline{B_n} \setminus B_{1/n}$ for $n$ an integer $\geq 2$. Arzela-Ascoli and the bound $ \| \tilde{u_k} \|_{2, \alpha, \delta} = 1$ imply that we can take a subsequence $ \tilde{u_k}^{(n)}$ which converges in $C^2(K_n)$ to some function $u_n$ such that $\triangle u_n =0$. The diagonal subsequence $ \tilde{u_n}^{(n)}$ converges to a function $u$ in $ \mathbb{C}^2 \setminus \lbrace 0 \rbrace$ which is in $C^2_{\delta}$ and $\triangle u =0$. Since $| \tilde{u_k} (\tilde{x_k})| \geq \epsilon$ we see that $u \not= 0$, but this contradicts Lemma \ref{inj}
		
	\end{proof}

	In practice we will only use the estimate \ref{Est 2} for functions $u$ with support outside $B_1$. For these functions we can give another equivalent definition of the norms \ref{norm1} and \ref{norm2}. Slightly abusing notation let us set
	\begin{equation} \label{norm4}
	\| f\|_{\alpha, \gamma} = \|f\|_{0, \gamma} + [f]_{\alpha, \gamma - \alpha}
	\end{equation}
	for functions $f$ with supp($f$) $\subset B_1^c$ ,  where
	$$[f]_{\alpha, \gamma - \alpha} = \sup_{x, y} \min \lbrace r(x), r(y) \rbrace^{-\gamma + \alpha} \frac{|f(x) - f(y)|}{d(x, y)^{\alpha}}$$
	when $\gamma <0$. (If $\gamma >0$ we replace $\min \lbrace r(x), r(y) \rbrace$ by  $\max \lbrace r(x), r(y) \rbrace$.)
	
	\begin{claim}
		\ref{norm1} and \ref{norm4} define equivalent norms
	\end{claim}
	
	\begin{proof}
		We prove that \ref{norm4} is bounded by a constant times \ref{norm1}.  Consider the case of $\gamma <0$. Take $x, y \in \mathbb{C}^2$ with $r(x) \leq r(y)$ such that
		$$ (1/2) [f]_{\alpha, \gamma - \alpha} \leq r(x)^{-\gamma + \alpha} \frac{|f(x) - f(y)|}{d(x,y)^{\alpha}} . $$
		Assume first that $r(y) \geq (5/4) r(x)$, say.  Then $d(x,y) \geq d(y,0) - d(x,0) \geq (1/4) r(x)$,  so that
		$$ (1/2) [f]_{\alpha, \gamma - \alpha} \leq r(x)^{-\gamma} |f(x)| + r(x)^{-\gamma} |f(y)| . $$
		When $\gamma < 0$,   $ r(x)^{-\gamma} |f(y)| \leq  r(y)^{-\gamma} |f(y)|$ and this last term is bounded by \ref{norm1}. When $ r(y) \leq (5/4) r(x)$ we write $ x=(r(x), \theta)$ and $y=(r(y), \psi)$. Let $\tilde{x} = (3/2, \theta)$ and $\tilde{y} = (\frac{3r(y)}{2r(x)}, \psi)$. Set $\lambda = (2/3)r(x)$ so that $D_{\lambda} (\tilde{x}) = x $ and  $D_{\lambda} (\tilde{y}) = y $. Note that $\tilde{x}, \tilde{y} \in A_1$ ($r(\tilde{y}) \leq 15/8 <2$), so that \ref{norm1} gives us a bound for
		$$ \lambda^{-\gamma} \frac{|f(x) - f(y)|}{d(\tilde{x}, \tilde{y})^{\alpha}} = (2/3)^{-\gamma} r(x)^{-\gamma + \alpha} \frac{|f(x)-f(y)|}{d(x,y)^{\alpha}} . $$ 
		From this we get  that \ref{norm4} is bounded by a constant times \ref{norm1}. The reverse inequality follows similarly.
	\end{proof}

	Finally let us point out that $(-2, 0) \cap I= \phi$ independently of $\overline{g}$. In fact, for this range one can give an alternative proof of \ref{Est 2} which does not invoke the spectrum of $\triangle_{\overline{g}}$.

	\begin{lemma} \label{DEC}
		Let $u \in C^2$   with supp($u$) $\subset B_1^c$. Assume $\triangle u = f \in C^0_{\delta-2}$ for some $\delta \in (-2, 0)$ and that $u \in C^0_{\mu}$ for some $\mu <0$. Then
		$$ \| u \|_{0, \delta} \leq c_{\delta} \| f \|_{0, \delta - 2}$$
		with $c_{\delta} = - (\delta +2)^{-1} \delta^{-1}$.	
	\end{lemma}
	
	\begin{proof}
		From \ref{lap} we have that $\triangle r^{{\delta}} = ({\delta} + 2){\delta}  r^{{\delta}-2}= -Q_{{\delta}} r^{{\delta}-2}$ with $ Q_{{\delta}} = - ({\delta} + 2) {\delta}  >0$. On $U_R = B_R \setminus B_1$  consider the function $ h = u - A r^{{\delta}} - m_R$ where $m_R = \sup_{\partial B_R} u$ and $ A= \|f\|_{0, \delta -2} / Q_{{\delta}}$. Then
		$$ \triangle h = f + \|f\|_{0, \delta -2} r^{{\delta}-2}  \geq 0 . $$
		$h \leq 0 $ on $\partial B_1$ since $u$ has support outside $B_1$. By our choice of $m_R$,  $h \leq 0$ on $\partial B_R$ . The maximum principle implies that $h\leq 0$ in $U_R$, i.e. for every $ x \in U_R$ we have that 
		$$u(x) \leq (\|f\|_{0, \delta -2}/ Q_{{\delta}}) r(x)^{{\delta}} + m_R . $$
		Since $u \in C^0_{\mu}$ for some $\mu <0$ we get that $\lim_{R\to \infty} m_R=0$. We let $R \to \infty$ and get the desired upper bound on $u$. The lower bound, and hence the lemma,  follows by applying the upper bound to $-u$.
	\end{proof}
	
	We explain the use of the maximum principle in the context of metrics with cone singularities -see also \cite{Jeffres}-. Let $ A = B_{R_2} \setminus B_{R_1} \subset \mathbb{C}^2 \setminus \lbrace 0 \rbrace$. Let $ h \in C^2(A)$ be such that $\triangle h \geq 0$ and $h|_{\partial A} \leq 0$. We claim that $h \leq 0$ on $A$, if this was not the case we can find $p \in A$ such that $h(p) = \sup_{A} h = 2m >0$. If $p \notin  L$ this would contradict the usual maximum principle. Then $ p \in L$. Let $\epsilon < \beta$ and $\delta$ be small enough such that $ \delta | P_d |^{2\epsilon} \leq m $  on $\partial A$. Consider the function $ H= h + \delta | P_d|^{ 2 \epsilon}$. By our choices $H$ has a local maximum at some point $q \in A$. Since $i \partial \overline{\partial} |P_d|^{2\epsilon} \geq 0$ we still have $\triangle H \geq 0$. Since $ \epsilon < \beta$ and $h$ is a $C^1$ function, we have that $ q \notin L$, contradicting the usual maximum principle. In fact this argument can be adapted to other situations. For example the same holds if $h$ is  $C^{\alpha}$, smooth outside $L$ with $\triangle h \geq 0$ (one then needs to take $\epsilon < \alpha \beta$).

	\subsection{Main result} \label{mr}
	
	We study the mapping properties  of the Laplacian of the reference metric $\omega$ of Lemma \ref{AM} acting on weighted spaces.

	We define our weighted H\"older spaces. The notation is the one of Subsection \ref{D}. Fix  $N$ large enough such that $ C \cap B_N^c \subset U_{2R, \delta/2} $. Let $\chi$ be a smooth function equal to $1$ on $B_{N+1}^c$  which vanishes on $B_N$.   For a function $u: \mathbb{C}^2 \to \mathbb{R}$ we write $u_{\infty} = \chi u \circ G$. We change notation and introduce a $'$ on the norms of the previous subsection. The space $C^{2, \alpha}_{\delta}$ ($C^{\alpha}_{\gamma}$) is defined to be the set of functions $u$ ($f$) such that the norm
	
	\begin{equation}
	\| u \|_{2, \alpha, \delta} = \| u \|_{C^{2, \alpha}(B_{N+1})} + \| u_{\infty}\|_{2, \alpha, \delta}'
	\end{equation}
	
	\begin{equation}
	\| f \|_{\alpha, \gamma} = \| f \|_{C^{\alpha}(B_{N+1})} + \| f_{\infty}\|_{\alpha, \gamma}'
	\end{equation}
	is finite. These are Banach spaces.
	
	Write $\triangle$ for the Laplacian of $\omega$. We apply our previous estimates to get the following
	
	\begin{corollary}
		Let $\delta \notin I$ and $\alpha < \beta^{-1} -1$. Then there exist a compact set $K$ and a constant $C$ such that for all $u \in C^{2, \alpha}_{\delta}$ with $\triangle u = f$ we have
		\begin{equation} \label{ESTIMATE}
		\|u\|_{2, \alpha, \delta} \leq C \left( \|u \|_{C^0(K)} + \| f \|_{\alpha, \delta -2} \right) . 
		\end{equation}
	\end{corollary}
	
	\begin{proof}
		The key point is that if $v \in (C^{2, \alpha}_{\delta})'$ is supported in $B_L^c$ then 
		\begin{equation} \label{ASYL}
		\| \triangle_{G^{*}g} v - \triangle_F v \|'_{\alpha, \delta -2} \leq c_L \| v \|'_{2, \alpha, \delta}
		\end{equation}
		with $c_L \to 0$ as $L \to \infty$, where $g$ is the metric corresponding to $\omega$ and $\triangle_F$ is the Laplacian of the flat metric. Since $G^{*} g = g_F$ in a region $U_{\delta' , R'}$ and $ |G^{*} g - g_F|_{g_F}= O(r^{\mu})$ for some $\mu < 0$ with derivatives on the complement of  $U_{\delta' , R'}$, \ref{ASYL} holds. The  corollary then follows from \ref{Est 2} and the interior estimates.
	\end{proof}

	\begin{lemma} \label{fred}
		$\triangle : C^{2, \alpha}_{\delta} \to C^{\alpha}_{\delta-2}$ has finite dimensional kernel for any $\delta$ and closed image when $\delta \notin I$.
	\end{lemma} 
	
	\begin{proof}
		Let us start by proving the statement about the kernel. Assume first that $\delta \notin I$ and let $u_k \in C^{2, \alpha}_{\delta}$ with $\triangle u_k = 0$ and $\| u_k \|_{2, \alpha, \delta} = 1$. By Arzela-Ascoli we can take a subsequence which converges in $C^0(K)$ to some function.  We apply the estimate \ref{ESTIMATE} to conclude that the subsequence is Cauchy in $C^{2, \alpha}_{\delta}$ and hence ker($\triangle$) is finite dimensional. In the case that $\delta \in I$ just take $\tilde{\delta} > \delta$, $\tilde{\delta} \notin I$ and note that $ C^{2, \alpha}_{\delta} \subset C^{2, \alpha}_{\tilde{\delta}}$.
		
		To prove that the image is closed let us write $ C^{2, \alpha}_{\delta} = V \oplus \mbox{ker}(\triangle)$ for some closed subspace $V$. We claim that there exists a constant $C$ such that $ \| u \|_{2, \alpha, \delta} \leq C \| f \|_{\alpha, \delta -2}$ for every $u \in V$. If this was not true then we would get a sequence such that $\| u_k \|_{2, \alpha, \delta} =1$ and $\|f_k \|_{\alpha, \delta -2} \to 0$. It follows from Arzela-Ascoli and \ref{ESTIMATE} that, after taking a subsequence, we can assume that $u_k$ converges in $C^{2, \alpha}_{\delta}$ to some function $u$ with $\triangle u =0$. Since $u \in V$ then $u=0$ and this contradicts $\|u_k\|_{2, \alpha, \delta} =1$. Finally let $ f_k = \triangle u_k $ with $ f_k \to f$ in $C^{\alpha}_{\delta-2}$. We can assume that $u_k \in V$. The estimate we just proved implies that $ \lbrace u_k \rbrace$ is Cauchy and converges to some $u \in C^{2, \alpha}_{\delta}$ with $\triangle u = f$.
		
	\end{proof} 
	
	Let $\mathcal{H}$ be the the completion of the space of compactly supported functions $\phi$, smooth in the cone coordinates,  under the Dirichlet norm $\int_{\mathbb{C}^2} |\nabla^{\omega} \phi |^2 \omega^2$. The Sobolev inequality -see Subsection \ref{SOBOLEV}- tells us that  there exists $C>0$ such that
		\begin{equation} \label{sob ineq}
		\left( \int |\phi|^4 \omega^2 \right)^{1/2} \leq C \int |\nabla^{\omega} \phi |^2 \omega^2
		\end{equation}
	for every $\phi \in \mathcal{H}$.

	Let $f \in L^{4/3}$. It follows from \ref{sob ineq} that $T_f (\phi) = \int f \phi$ defines a bounded functional on $\mathcal{H}$. A weak solution of $\triangle u = f$ is a function $ u \in \mathcal{H}$ such that $ - \int \langle \nabla u, \nabla \phi \rangle = \int f \phi$ for every $\phi \in \mathcal{H}$. It follows from \cite{Donaldson} that if $f$ is locally in $C^{\alpha}$ then $u$ is locally in  $C^{2, \alpha}$.
	
	\begin{lemma} \label{dense}
		Let $f \in C^{\alpha}_{c}$ and $u \in \mathcal{H}$ be a weak solution of $\triangle u = f$. Then $u \in C^{2, \alpha}_{\delta}$ for any $\delta > -2$
	\end{lemma}
	
	\begin{proof}
		Take $ \psi = \psi(t)$ to be a smooth non-decreasing function of one real variable with $\psi(t) = t$ when $t \geq 2$ and $\psi(t) =1$ when $t \leq 1$. Define $\rho = \psi \circ r$ and let
		$$ \| u \|^2_{L^2_{\delta}} = \int |u|^2 \rho^{-2\delta} \rho^{-4} . $$
		Since $u \in \mathcal{H}$ we get that $ \int |u|^4$ is finite (in fact it is bounded by $\|f\|_{L^{4/3}}$). From H\"older's inequality we have that 
		$$ \| u \|_{L^2_{\delta}}^2 \leq  \left( \int |u|^4 \right)^{1/2} \left( \int \rho^{-4(\delta + 2)} \right)^{1/2}.$$
		If $\delta > -1$ we conclude that $ \| u \|_{L^2_{\delta}}$ is finite.
		
		In the interior Schauder estimates one can replace the $C^0$ norm in the r.h.s with the $L^2$ norm -see the proof of Proposition 2.4 in \cite{Brendle} and Theorem 8.17 in \cite{GilbargTrudinger}-. Using the interior estimates in this form one gets that if $u$ is locally in $C^{2, \alpha}$ and $ \|u\|_{L^2_{\delta}}$ is finite, then $u \in C^{2, \alpha}_{\delta}$ and
		
		$$ \| u \|_{2, \alpha, \delta} \leq C \left( \| f \|_{\alpha, \delta-2} + \| u \|_{L^2_{\delta}} \right) . $$
		Hence $u \in C^0_{\delta}$ for any $\delta > -1$. One can then use Lemma \ref{DEC} to show that in fact this is true for any $\delta > -2$.
		
	\end{proof}
	
	\begin{proposition} \label{LTHRM}
		$\triangle : C^{2, \alpha}_{\delta} \to C^{\alpha}_{\delta-2}$ is an isomorphism when $\delta \in (-2, 0)$ and is surjective when $\delta \in (0, 2) \setminus I$.
	\end{proposition}
	
	\begin{proof}
		The fact that $\triangle$ is injective when $\delta <0$ follows from the maximum principle or by integration by parts. The key is to prove that the map is onto.  By lemma \ref{fred}  it is enough to prove that the image is dense. We know from lemma \ref{dense} that the space of $C^{\alpha}$ functions with compact support is contained in the image, one detail is that this space is not dense in $C^{\alpha}_{\delta-2}$. But this can be overcome as follows:  Take $f \in C^{\alpha}_{\delta-2}$ and $\delta < \tilde{\delta} < 2$ with $[\delta, \tilde{\delta}] \cap I = \phi$ . Let $h_n$ be a sequence of smooth cut-off functions with $h_n = 1$ on $B_n$ and $h_n = 0$ on $B_{n+1}^c$. The sequence of functions $ f_n = h_n f \to f $ in $C^{\alpha}_{\tilde{\delta}-2}$ so that we can find $u \in C^{2, \alpha}_{\tilde{\delta}}$ with $\triangle u = f$. It follows from the proof of lemma \ref{fred} that we can take $u \in C^{2, \alpha}_{\delta'}$ for any $\delta' \in (\delta, \tilde{\delta}]$ and with $ \|u\|_{2, \alpha, \delta'} \leq C \|f\|_{\alpha, \delta}$ with $C$ independent of $\delta'$. By taking the limit as $\delta' \to \delta$ we get that $u \in C^{2, \alpha}_{\delta}$
		
	\end{proof}
	
	\begin{remark} \label{LRMK}
		If $\omega_u = \omega + i \partial \overline{\partial} u$ is a K\"ahler metric on $\mathbb{C}^2 \setminus C$ with $u \in C^{2, \alpha}_{\delta}$ for some $\delta < 2$, then Proposition \ref{LTHRM} holds for the Laplacian of $\omega_u$.
	\end{remark}
	
	Finally we mention some properties of these weighted spaces that will be useful to us later.
	
	\begin{itemize}
		
		\item Multiplication gives a bounded map
		$$ C^{\alpha}_{\gamma_1} \times C^{\alpha}_{\gamma_2} \to C^{\alpha}_{\gamma_1 + \gamma_2}$$
		
		\item Let $\lbrace f_j \rbrace_{j=1}^{\infty} \subset C^{\alpha}_{\gamma}$ with $\|f_j\|_{\alpha, \gamma} \leq C$ for some constant $C$. Then, after taking a subsequence,  we can assume that $f_j \to f$ uniformly in compact subsets to some function $f$. Moreover $ f \in C^{\alpha}_{\gamma}$ and $\| f \|_{\alpha, \gamma} \leq C$.
		
		\item Let $f \in C^{\tilde{\alpha}}_{\tilde{\gamma}}$ and $\alpha < \tilde{\alpha}$, $ \tilde{\gamma} < \gamma$. Then for every $\epsilon >0$ we can find $ h \in C^{\infty}_c$ such that $\| f - h \|_{\alpha, \gamma} < \epsilon$.
		
	\end{itemize}

	\subsection{Application} \label{ap}
	
	We use Proposition \ref{LTHRM} and the implicit function theorem to prove the existence of a metric $\omega_0$ with bounded Ricci curvature. In fact $\omega_0$ is Ricci-flat outside a compact set. It is not hard to see that the metrics $\omega$ and $\omega_B$ constructed in Section \ref{Rmetrics} have unbounded Ricci curvature. One can easily adapt the proof of Proposition \ref{BR} to show that in the general setting of a compact K\"ahler manifold with a smooth divisor $D \subset X$ there are metrics with cone singularities along $D$ and bounded Ricci curvature. 
	
	\begin{proposition} \label{BR}
		There exists $u_0 \in C^{2, \alpha}_{\delta}$ for some $\delta <2$ such that $\omega_0 = \omega + i \partial \overline{\partial} u_0$ is a K\"ahler form on $\mathbb{C}^2 \setminus C$ with
		$$ \omega_0^2 = e^{-f_0} |P|^{2\beta-2} \Omega \wedge \overline{\Omega}$$
		and $f_0 \in C^{\infty}_c$.
	\end{proposition}
	
	\begin{proof}
		Write 
		$$ \omega^2= e^{-f} |P|^{2\beta -2} \Omega \wedge \overline{\Omega} . $$
		We claim that there exists  $ 0 < \tilde{\alpha} < \beta^{-1} -1$ and $\tilde{\gamma} < 0$ such that $ f \in C^{\tilde{\alpha}}_{\tilde{\gamma}}$. The fact that $f \in C^{\tilde{\alpha}}$ on compacts subsets follows from the local expression \ref{DEF}. Lemma \ref{V} then proves the claim. (We can take any $\tilde{\gamma} > -1/c$ .)
		Let $ 0 < \alpha < \tilde{\alpha}$ and $ \tilde{\gamma} < \gamma <0$ such that $\delta = \gamma + 2 \notin I$.  Then there exist $\lbrace h_j \rbrace_{j=1}^{\infty} \subset C^{\infty}_{c}$ such that $\lim_{j \to \infty} \| f - h_j \|_{\alpha, \gamma}=0$.  
		
		Consider the bounded map $\mathcal{F} : U \subset C^{2, \alpha}_{\delta} \to C^{\alpha}_{\delta-2}$ defined in a neighborhood of $0$ and given by
		
		$$ \mathcal{F} (u) = \log \frac{(\omega + i \partial \overline{\partial} u)^2}{\omega^2} .$$
		So that $\mathcal{F}(0) = 0$ and $D\mathcal{F}|_0 = \triangle$. By \ref{LTHRM} and the Implicit Function Theorem,  we can solve $\mathcal{F}(u_0) = f -h_N$ for some $N>>1$.  We get that
		$$ (\omega + i \partial \overline{\partial} u_0)^2 = e^{f-h_N} \omega^2$$
		and the proposition is proved with $f_0 = h_N$.
		
		Note that the positivity of $\omega_0$ follows from the equation that its volume form satisfies together with the decay of $i\partial\overline{\partial} u_t$ and the connectedness of $\mathbb{C}^2 \setminus C$.
		
	\end{proof}
	
	In Proposition \ref{BR} the function $f_0$ is smooth with respect to the complex coordinates. One can use the same proof, with the obvious modification, to get a K\"ahler metric $\tilde{\omega}_0$ such that $ \tilde{\omega}_0^2 = e^{-\tilde{f}_0} |P|^{2\beta-2} \Omega \wedge \overline{\Omega}$, with $\tilde{f}_0$ a compactly supported function, smooth in the cone coordinates. What will be relevant for us in the next section is that $i\partial \overline{\partial} f_0$ is bounded with respect to $\omega_0$; the metric $\tilde{\omega}_0$ would do the job as well.
	
	\section{A priori estimates for the Monge-Ampere equation} \label{APRIORI}
	
	Let  $\omega_0$ be given by Proposition \ref{BR}, so that  
	$$\omega_0^2 = e^{-f_0} |P|^{2\beta -2} \Omega \wedge \overline{\Omega} $$
	with $f_0 \in C^{\infty}_c$.
	Fix $ 0 < \alpha < \beta^{-1} -1$ and $ -2 < \delta <0$. The main result of this section is the following

	\begin{proposition} \label{a priori est}
		There exists a constant $C$ independent of $ t \in [0, 1]$ such that if $ u_t \in C^{2, \alpha}_{\delta}$ solves
		$$ (\omega_0 + i \partial \overline{\partial} u_t)^2 = e^{tf_0} \omega_0^2 , $$
		then $\|u_t\|_{2, \alpha, \delta} \leq C$.
	\end{proposition}
	In the next subsections we derive a priori estimates on different norms of $u_t$ which can be stated in the same form as Proposition \ref{a priori est}. To avoid repetition we only state the estimate proved. We simplify notation and write $ f = t f_0$ and $u=u_t$. We hope that this simplified notation doesn't cause any confusion at the end. The set up for this section is then  a smooth function with compact support $ f$ and $ u \in C^{2, \alpha}_{\delta}$  a solution of
	\begin{equation} \label{CMA}
	(\omega_0 + i \partial \overline{\partial} u)^2 = e^{f} \omega_0^2 . 
	\end{equation}
	We denote by $\omega_u$ the corresponding K\"ahler form $\omega_0 + i \partial \overline{\partial} u$.
	
	\subsection{$C^0$ estimate}
	
	\begin{proposition} \label{C0EST}
		$ \| u \|_0 \leq C$ .
	\end{proposition}
	We follow \cite{Joyce} (pages 188-190).  The technique is Moser iteration.  Note that $u \in C^0_{\delta}$ implies that $u \in L^p$ for $p$ large and $\|u\|_0 = \lim_{p\to \infty} \|u \|_{L^p}$. The proof of Proposition \ref{C0EST} begins with the following
	
	\begin{lemma}
		Let $p>2$ with $p\delta +2 <0$. Write  $\phi = u |u|^{p/2 -1}$.  Then  we have
		\begin{equation}\label{int}
		\int_{\mathbb{C}^2} |\nabla \phi|^2 \omega_0^2 \leq p \int_{\mathbb{C}^2} u |u|^{p-2} (1-e^f) \omega_0^2 .
		\end{equation}
	\end{lemma}
	
	\begin{proof}
		We have to check that the relevant integration by parts arguments in \cite{Joyce} hold in our context of metrics with cone singularities.
		Define a 3-form $\eta$ by
		$$ \eta = u |u|^{p-2} i \overline{\partial} u \wedge ( \omega_0 + \omega_u) , $$
		where $\omega_u = \omega + i \partial \overline{\partial}u$.  This form is $C^1$ on $\mathbb{C}^2 \setminus C$. Fix $R,  \epsilon>0$ and consider the region $U = B_R \setminus \lbrace |P| \leq \epsilon \rbrace$. By Stokes' theorem $ \int_U d\eta = \int_{\partial U} \eta$. Use the equation \ref{CMA} to get 
		$$ d\eta = (p-1) |u|^{p-2} i \partial u \wedge \overline{\partial} u \wedge ( \omega_0 + \omega_u) + u |u|^{p-2} (e^f -1) \omega_0^2 . $$
		For $R$ fixed we let $\epsilon \to 0$. Write $ C_{\epsilon} = \lbrace |P| = \epsilon \rbrace \cap B_R$. Note that $\lim_{\epsilon \to 0} \mbox{Area}_{g_0} (C_{\epsilon}) = 0$ and that $|\eta|_{g_0} $ is bounded. We conclude that we can take $U=B_R$ and $\partial U = S_R$. Now note that $ \mbox{Vol}_{g_0} (S_R) \leq C R^3$ and $ | \eta|_{g_0} \leq C R^{(p-1) \delta + \delta -1}$ on $S_R$. The choice $ p \delta < -2$ gives $ \lim_{R \to \infty} \int_{S_R} \eta = 0$ and we get $\int_{\mathbb{C}^2} d\eta =0$.
		The lemma follows from $ i \partial u \wedge \overline{\partial} u \wedge \omega_0 = | \nabla u |^2 \omega_0^2$, $|\nabla \phi |^2= (p^2/4) |u|^{p-2} | \nabla u |^2 $ and $ i \partial u \wedge \overline{\partial} u \wedge \omega_u = F \omega_0^2$ with $F = |\nabla u |^2_{g_u} (\omega_u^2 / \omega_0^2) \geq 0$.
		
	\end{proof}
	
	Now we prove Proposition \ref{C0EST}.
	
	\begin{proof}
		The Sobolev inequality for the metric $\omega_0$ tells us 
		\begin{equation} \label{sob}
		\left( \int_{\mathbb{C}^2} |\phi|^4  \omega_0^2 \right)^{1/2} \leq C  \int_{\mathbb{C}^2} | \nabla \phi|^2 \omega_0^2 .
		\end{equation}
		Apply this to $\phi = u |u|^{p/2-1}$ and  use \ref{int} to get 
		\begin{equation} \label{induction}
		\| u \|^p_{L^{2p}} \leq C p \| u \|^{p-1}_{L^{p-1}} .
		\end{equation}
		The next step is to estimate $\| u \|_{L^{p_1}}$ for some $p_1 > 2$. In order to do this we fix some $p_0 > 2$ such that $p_0 \delta +2 < 0$ .  Use \ref{int}, \ref{sob}  to get
		$$ \left( \int |u|^{2p_0} \omega_0^2 \right)^{1/2} \leq p_0 \int |1-e^f| |u|^{p_0-1} \omega_0^2 . $$
		Let $r>1$ be given by $r(p_0 -1) = 2p_0$ and $q$ by $r^{-1} + q^{-1} =1$. Let $\rho$ be a function $\geq 1$ that agrees with $r$ outside a compact set, as in the proof of Lemma \ref{dense}. We replace $|1-e^f| \leq C \rho^{\gamma}$, with $\gamma = \delta -2$. From the choices it follows that $\| \rho^{\gamma} \|_{L^q} \leq C$. H\"older's inequality then implies that $\| u \|_{L^{p_1}} \leq C$ with $p_1 = 2p_0$.
		Using the bound on $\| u \|_{L^{p_1}}$, \ref{induction} and an induction argument we get a uniform bound (independent of $p$) on $\| u \|_{L^{p}}$. Finally  $\| u \|_{C^0} = \lim_{p \to \infty} \| u \|_{L^p} \leq C$.
	\end{proof}

	\subsection{$C^2$ estimate}
	
	\begin{proposition} \label{C2EST} $C^{-1} 
		\omega_0 \leq \omega_u \leq C \omega_0$.
	\end{proposition}
	To prove Proposition \ref{C2EST}  we use the maximum principle. Our main tool is the Chern-Lu inequality (Lemma \ref{CL} below). In Yau's proof of the Calabi conjecture, the constant $C$ in Proposition \ref{C2EST} depends on a lower bound on the bisectional curvature of a reference metric. In our case we don't know of any reference metric with bisectional curvature bounded from below and there might be obstructions to the existence of one. The use of the Chern-Lu inequality (Lemma \ref{CL}) allows us to overcome this problem. Our methods in this subsection are highly inspired by Jeffres-Mazzeo-Rubinstein \cite{JMR}, although we use the Chern-Lu inequality in a slightly different way than in \cite{JMR}.
	
	\begin{lemma} \label{CL}
		Let $g$ and $\hat{g}$ be two K\"ahler metrics on $X$ such that $\mbox{Ric}(g) \geq -Q_2 \hat{g}$ and $\mbox{Bisec}(\hat{g}) \leq Q_1$ for some $Q_1, Q_2 >0$. Set $\phi = \mbox{tr}_{g} (\hat{g})$. Then
		\begin{equation}
		\triangle_{g} \log \phi \geq - Q \phi , 
		\end{equation}
		where $Q=Q_1 + Q_2$.
	\end{lemma}
	
	For the sake of completeness we include a proof of Lemma \ref{CL}. Before going to the proof we mention two points:
	
	\begin{itemize}
		\item Lemma \ref{CL} is a particular case of Proposition 7.1 in \cite{JMR}. 
		\item There is a similar formula for $\triangle_{\hat{g}} \log \phi $ if we assume an upper bound on the Ricci curvature of $\hat{g}$ and a lower bound on the bisectional curvature of $g$. See Chapter 3 in \cite{Szekelyhidi}.
	\end{itemize}
	
	\begin{proof} 
		Let $ x \in X$ and $ (z_1, \ldots, z_n)$ be holomorphic coordinates around $x$. The metrics are then given by $n$ by $n$ Hermitian matrices $(g_{i\overline{j}})$ and $(\hat{g}_{i\overline{j}})$. At the point $x$ we require that $(g_{i\overline{j}})$ is diagonal ,   $(\hat{g}_{i\overline{j}})$ is the identity and all the first derivatives of $(\hat{g}_{i\overline{j}})$ vanish.   The function $\phi$ is then given by   
		$$\phi = \mbox{tr}_{\omega} \hat{\omega} = \sum_{j, k} g^{j\overline{k}} \hat{g}_{j\overline{k}} , $$
		where $(g^{i\overline{j}})$ is the inverse transpose of $(g_{i\overline{j}})$.
		First we compute $\triangle_g \phi $. At the point $x$ we have
		$ \triangle_g \phi = \sum_{p} g^{p \overline{p}} \partial_{\overline{p}} \partial_p \phi $, where
		$$ \partial_{\overline{p}} \partial_p \phi = \sum_{j, k} (\partial_{\overline{p}} \partial_p g^{j\overline{k}}) \hat{g}_{j\overline{k}} + (\partial_{\overline{p}} \partial_p \hat{g}_{j\overline{k}}) g^{j\overline{k}} =  \sum_{j} \partial_{\overline{p}} \partial_p g^{j\overline{j}}  + (\partial_{\overline{p}} \partial_p \hat{g}_{j\overline{j}}) g^{j\overline{j}} .  $$
		We write
		$ \triangle_{\omega} \phi = I + II$ with
		$$ I = \sum_{p, j} g^{p\overline{p}} g^{j \overline{j}} \hat{g}_{j\overline{j},  p\overline{p}} , \hspace{4mm} II =  \sum_{p, j} g^{p\overline{p}} \partial_{\overline{p}} \partial_p g^{j\overline{j}} .  $$
		Subindices  after the comma indicate differentiation.  Since the coordinates are adapted to $\hat{g}$ at $x$ and the bisectional curvature of $\hat{g}$ is bounded from above by $Q_1$; we have that for every $p$ and $j$,  $-\hat{g}_{j\overline{j},  p\overline{p}} \leq Q_1$. It follows that
		\begin{equation} \label{CL1}
		I \geq - Q_1  \sum_{p, j} g^{p\overline{p}} g^{j\overline{j}} = -Q_1 \phi^2 .
		\end{equation}
		Since $(g_{i\overline{j}})$ is diagonal at $x$, we get
		$$ II  = - \sum_{p, j} g^{p\overline{p}} (g^{j\overline{j}})^2  g_{j\overline{j}, p \overline{p}} . $$
		Denote by $R$ the Riemann curvature tensor of $g$,  given by
		$$ R_{i\overline{j}k\overline{l}} = -g_{i\overline{j}, k\overline{l}} + \sum_{r, s} g^{r\overline{s}} g_{i\overline{s}, k} g_{r\overline{j}, \overline{l}}.$$
		We conclude that
		$$ - g_{l\overline{q}, p\overline{p}} = R_{l\overline{q}p\overline{p}} - \sum_{r, s} g^{r\overline{s}} g_{l\overline{s}, p} g_{r\overline{q}, \overline{p}} . $$
		We obtain
		\begin{equation} \label{CL2}
		II = \sum_{j, p} g^{p\overline{p}} (g^{j \overline{j}})^2 R_{j \overline{j} p \overline{p}}  + P  ,
		\end{equation}
		where
		$$ P = \sum_{p, j, r} (g^{j\overline{j}})^2  g^{r\overline{r}} g^{p\overline{p}} | g_{j\overline{r}, p}|^2 . $$
		Note that $\mbox{Ric}_{j\overline{j}} = \sum_{p} g^{p\overline{p}} R_{j\overline{j}p\overline{p}}$ is the Ricci curvature of $g$. Since $\mbox{Ric}(g) \geq -Q_2 \hat{g}$ we can bound first term in \ref{CL2} by 
		\begin{equation} \label{CL3}
		\sum_{j, p} g^{p\overline{p}} (g^{j \overline{j}})^2 R_{j \overline{j} p \overline{p}} =  \sum_j  (g^{j \overline{j}})^2 \mbox{Ric}_{j\overline{j}} \geq  -Q_2 \phi^2. 
		\end{equation}

		Now we bring in the logarithm to get
		$$ \triangle_g \log \phi = \frac{\triangle_g \phi}{\phi} - \frac{|\nabla \phi|^2}{\phi^2} . $$
		It follows from \ref{CL1}, \ref{CL2} and \ref{CL3} that to prove the lemma it is enough to show that $ | \nabla \phi |^2 \leq \phi P $. At the point $x$ we have
		$$ | \nabla \phi |^2 = \sum_p g^{p \overline{p}} | \partial_{p} \phi|^2, \hspace{4mm}  \partial_p \phi = - \sum_j (g^{j\overline{j}})^2 g_{j\overline{j}, p} .  $$
		So that
		$$ | \nabla \phi |^2 =  \sum_{p, j, r} g^{p\overline{p}} (g^{j\overline{j}})^2 g_{j\overline{j}, p} (g^{r\overline{r}})^2 g_{r\overline{r},\overline{p}}  =  \sum_{j, r} (g^{j\overline{j}})^2 (g^{r\overline{r}})^2 \left( \sum_p g^{p\overline{p}}  g_{j\overline{j},p}  g_{r\overline{r},\overline{p}} \right) . $$
		For each $j$ and $r$ fixed the Cauchy-Schwarz inequality implies that
		$$ \sum_p g^{p\overline{p}}  g_{j\overline{j},p}  g_{r\overline{r},\overline{p}} \leq \left( \sum_p g^{p\overline{p}} |g_{j\overline{j}, p}|^2 \right)^{1/2}       \left( \sum_p g^{p\overline{p}} |g_{r\overline{r}, \overline{p}}|^2 \right)^{1/2} .  $$
		We use the Cauchy-Schwarz inequality once again to obtain
		$$ | \nabla \phi |^2 \leq \left( \sum_j  (g^{j\overline{j}})^2 \left( \sum_p g^{p\overline{p}} |g_{j\overline{j}, p}|^2 \right)^{1/2} 
		\right)^2  =  \left( \sum_j  (g^{j\overline{j}})^{1/2} \left( \sum_p (g^{j\overline{j}})^3 g^{p\overline{p}} |g_{j\overline{j}, p}|^2 \right)^{1/2} \right)^2        $$
		
		$$ \leq \phi  \left( \sum_{j,p} (g^{j\overline{j}})^3 g^{p\overline{p}} |g_{j\overline{j}, p}|^2  \right) \leq  \phi  \left( \sum_{j, r, p} (g^{j\overline{j}})^2 g^{r\overline{r}} g^{p\overline{p}} |g_{j\overline{r}, p}|^2  \right)   = \phi P . $$
		The proof of the lemma is now complete.
		
	\end{proof}
	
	We are now ready to prove Proposition \ref{C2EST}.
	
	\begin{proof}
		We set $g$ and $\hat{g}$ to be the K\"ahler metrics corresponding to $\omega_u$ and $\omega_B$, respectively. First we check that the hypothesis of Lemma \ref{CL} hold. The upper bound on the bisectional curvature of $\hat{g}$ is given by Lemma \ref{UBD}. Recall that $\omega_u^2 = e^{tf_0} \omega_0^2$, where $\mbox{Ric}(\omega_0) = i \partial \overline{\partial} f_0$. It follows that $\mbox{Ric}(\omega_u)=(1-t) \mbox{Ric}(\omega_0)$. Since $f_0$ is smooth we clearly have $i\partial \overline{\partial} f_0 \geq -Q_2 \omega_B$ for some $Q_2 >0$. We conclude that the bound $\mbox{Ric} (g) \geq - Q_2 \hat{g}$ holds.
		
		Write $\omega_u = \omega_B + i \partial \overline{\partial} v$.  Note that $u$ and $v$ differ by a fixed function. Take the trace w.r.t. $\omega_u$ to get $ 2 = \phi + \triangle_g v$. Consider the function $ H = \log \phi - A v$, with $A = Q+1$. We want to show  that $H$ is bounded above by a uniform constant. Since $H(y) \to \log 2$ as $y \to \infty$, we can assume that $H$ attains its global maximum at $x \in \mathbb{C}^2$. If $x \notin C$, by Lemma \ref{CL} we have
		$$ 0 \geq \triangle_{g} H (x) \geq -Q\phi - A\triangle_{g} v = \phi (x) - 2A .$$
		Proposition \ref{C0EST} gives us a uniform bound on the $C^0$ norm of $u$ and hence of $v$. We conclude that at the point $x$ the function $H$ is bounded from above by a uniform constant. Since $x$ is a maximum point of $H$ the bound holds in all of $\mathbb{C}^2$.
		
		If $x \in C$ we can assume $H(x) \geq \log 2 +3$ and take $R>0$ so that $H|_{\partial B_R} \leq \log2 +1 $. Fix some $0<\epsilon < \beta$ and consider the function $\tilde{H} = H + (1/N) |P|^{2\epsilon}$,  where $N>0$ is big enough such that $ (1/N) |P|^{2\epsilon} \leq 1$ on $\partial B_R$. By our choices $\max_{y \in \overline{B_R}} \tilde{H} =\tilde{ H} (\tilde{x})$ with $\tilde{x} \notin \partial B_R$. Since $H \in C^{\alpha}$ and $\epsilon < \beta$, we have that $\tilde{x} \notin C$, hence
		$$ 0 \geq \triangle_{\omega} \tilde{ H} (\tilde{x}) = \triangle_{\omega} H + (1/N) \triangle_{\omega} |P|^{2\epsilon} \geq \triangle_{\omega} H (\tilde{x})    \geq \phi (\tilde{x}) - 2A$$
		We used that $\triangle_{\omega} |P|^{2\epsilon} \geq 0$ since $i\partial \overline{\partial}|P|^{2\epsilon} \geq 0$.  Note that $H(x) \leq \tilde{H}(x) \leq \tilde{H}(\tilde{x})$ to get the estimate.

		We have proved that $H$ is uniformly bounded from above. We use Proposition \ref{C0EST} once again to conclude that $ \phi = \mbox{tr}_g \hat{g} \leq \tilde{C}$. Therefore $\omega_B \leq \tilde{C} \omega_u$. Since the metrics $\omega_B$ and $\omega_0$ are fixed there is a fixed constant $\Lambda$ such that $\Lambda^{-1}\omega_0 \leq \omega_B$, hence $\omega_0 \leq \Lambda \tilde{C} \omega_u$. Finally we use the equation $\omega_u^2 = e^f \omega_0^2$ to get the desired bound $C^{-1} \omega_0 \leq \omega_u \leq C \omega_0$.

	\end{proof}

	\subsection{$C^{2, \alpha}$ estimate}
	
	\begin{proposition} \label{C2AEST}
		$ \| u \|_{2, \alpha} \leq C$.
	\end{proposition}
	Proposition \ref{C2AEST} is a direct consequence of Theorem 1.7 in \cite{ChenWang} which we record into Proposition \ref{CW} below. When $\beta =1$ the content of Proposition \ref{CW} is well-known and is a consequence of the so-called Evans-Krillov theory. We mention that in Yau's work on the Calabi conjecture Proposition \ref{C2AEST} was proved by means of the maximum principle -a step known as  Calabi's third order estimate-; while Proposition \ref{CW} is a local statement. Calabi's third order estimate was carried over to the context of metrics with cone singularities, under the assumption that $\beta \leq 1/2$, in Section 6  of \cite{Brendle}.
	
	We work on the space $\mathbb{C}_{\beta} \times \mathbb{C}^{n-1}$ with complex coordinates $z_1, \ldots, z_n$. If $p \in \mathbb{C}^n$ and $r>0$, we denote by $B_r (x)$ the metric ball  with center at $x$ and radius $r$ -in the distance induced by $g_{(\beta)}$. When $x=0$ we write $B_r = B_r(0)$.
	
	\begin{proposition} \label{CW}
		Let $\alpha < \alpha'<\beta^{-1} -1$ and $ \phi \in C^{2, \alpha'} (B_1)$ be such that
		$$ K^{-1} \omega_{(\beta)} \leq i\partial \overline{\partial} \phi \leq K \omega_{(\beta)}$$
		and
		$$ \det (\partial \overline{\partial} \phi) = |z_1|^{2\beta-2} e^f .$$
		Then there exists a constant $C$, which depends only only on $K$ and the $C^{\alpha}$ norm of $f$ in $B_1$ such that
		$$ [ i\partial \overline{\partial} \phi ]_{\alpha, B_{1/4}} \leq C . $$
	\end{proposition}
	For the sake of completeness we give a sketch of the proof of Proposition \ref{CW} in \cite{ChenWang}. The technique is known as `the blow-up argument' and was popularized in Geometric Analysis by Leon Simon \cite{Simon}.
	\begin{proof}
		There are three main ingredients:
		\begin{itemize}
			\item Fact 1:  Let $\Lambda >0$ and $\Lambda^{-1} < r < \Lambda$. There exists a constant $C$ which depends only on $\Lambda$ with the following property: If $\eta$ is a real closed $C^{\alpha}$ form on $B_r$ of type $(1, 1)$, then there exists a real function $\phi \in C^{2, \alpha}(B_{r/2})$ such that $i\partial\overline{\partial}\phi = \eta$ on $B_{r/2}$ and $\| \phi \|_{2, \alpha, B_{r/4}} \leq C \| \eta \|_{\alpha, B_r}$.
			\item Fact 2: Let $\omega_{\infty}$ be a $C^{\alpha}$ K\"ahler metric on $\mathbb{C}^n$ such that $\omega_{\infty}^n = \omega_{(\beta)}^n$ and $K^{-1} \omega_{(\beta)} \leq \omega_{\infty} \leq K \omega_{(\beta)}$ for some $K>0$. Then there exists a linear transformation $L$, which preserves $\lbrace z_1 =0 \rbrace$,  such that $\omega_{\infty} = L^{*} \omega_{(\beta)}$.
			\item Fact 3: There is a $\delta>0$ such that Proposition \ref{CW} holds when $K=1+\delta$.
		\end{itemize}
		Fact 1 is proved by analyzing the standard proof of the local $\partial \overline{\partial}$-lemma, see Section 4 in \cite{ChenWang2}. Fact 2 is proved by means of the maximum principle in Section 5 of \cite{ChenWang2}; in the case that $\omega_{\infty}$ is known to be a Riemannian cone see Proposition 25 in \cite{CDSII}. Fact 3 is proved in pages 228-229 of \cite{CDSII}, it follows from the interior Schauder estimates and the fact that the Laplace operator is the linearization of the Monge-Ampere operator.
		
		Now let us proceed with the proof of Proposition \ref{CW}. We write $\omega = i \partial \overline{\partial} \phi$. For $q \in B_1$ denote by $d_q$ the distance from $q$ to the boundary of $B_1$.  Define the H\"older radius, $h_{\omega, q}$, as the supremal of the $ h \in (0, d_q)$ such that $ [ \omega]_{\alpha, B_h (q)} \leq \delta_0 h^{-\alpha}$, where $\delta_0$ is a small positive number which depends only on $K$. To prove the proposition it is enough to show that there exists a constant $c_0>0$, depending only on $K$ and $\|f\|_{\alpha, B_1}$,  such that $h_{\omega, q}/d_q \geq c_0$ for all $q \in B_1$. We argue by contradiction and assume that there are $\omega_k = i \partial \overline{\partial} \phi_k$, $q_k \in B_1$ such that:
		\begin{equation} \label{SEQ1}
		K^{-1} \omega_{(\beta)} \leq \omega_k \leq K \omega_{(\beta)}, \hspace{3mm}  \det (\omega_k) = |z_1|^{2\beta-2} e^{f_k}, \hspace{3mm} \| f_k \|_{\alpha', B_1} \leq 1
		\end{equation}
		
		\begin{equation} \label{SEQ2}
		\frac{h_{\omega_k, q_k}}{d_{q_k}} = \epsilon_k \to 0, \hspace{3mm} \frac{h_{\omega_k, q_k}}{d_{q_k}} \leq 2 \inf_{q \in B_1} \frac{h_{\omega_k, q}}{d_{q}} .
		\end{equation}
		We rescale  and define
		$$ \hat{z_1} = h_{\omega_k, q_k}^{-1/\beta} (z_1 - z_1(q_k) ) , \hspace{3mm} z_j = h_{\omega_k, q_k}^{-1} (z_j - z_j(q_k)) \hspace{3mm} \mbox{for} \hspace{1mm} j =2, \ldots, n .$$
		Write this change of coordinates as $\hat{x} = \tilde{\Gamma}_k (x)$. Let $\Gamma_k$ be the inverse of $\tilde{\Gamma}_k$ and consider $\hat{\omega}_k = h_{\omega_k, q_k}^{-2} \Gamma_k^{*} \omega_k$, $\hat{f}_k = \Gamma_k^{*} f_k$. It follows that
		
		\begin{equation} \label{SEQ3}
		\det (\hat{\omega}_k) = |\hat{z}_1|^{2\beta -2} e^{\hat{f}_k}, \hspace{3mm} h_{\hat{\omega}_k, 0} =1 .
		\end{equation}
		Note that $\tilde{\Gamma}_k  (B_{d_{q_k}}(q_k)) = B_{1/\epsilon_k}(0)$, so that the $\omega_k$ are defined on larger and larger balls.
		It is not hard to show, by means of the Arzela-Ascoli theorem and a diagonal argument, that \ref{SEQ1}, \ref{SEQ2} together with Fact 3 and Fact 1; imply that $\hat{\omega}_k$ converges in $C^{\alpha}$, up to a subsequence,  to a K\"ahler metric $\hat{\omega}_{\infty}$ defined on $\mathbb{C}^n$ as the one in Fact 2. It follows that $\hat{\omega}_{\infty}$ has constant coefficients, so that $h_{\hat{\omega}_{\infty}, 0} = \infty$. Since $\hat{\omega}_k \to \hat{\omega}_{\infty}$ in $C^{\alpha}$, for $k$ large enough we get that $ h_{\hat{\omega}_k, 0} \geq 2$ and this contradicts \ref{SEQ3}.

	\end{proof}

	\subsection{Weighted estimates} \label{WESTIMATES}
	This subsection completes the proof of Proposition \ref{a priori est}. Our main results are Proposition \ref{W1} and Proposition \ref{W2}. Our  reference is Chapter 8 in Joyce's book \cite{Joyce} again.  Proposition \ref{W1} corresponds to Theorem 8.6.6 in \cite{Joyce} and Proposition \ref{W2} to Theorem 8.6.11 in \cite{Joyce}.
	
	The first result is a weighted version of Proposition \ref{C0EST}. The proof uses Moser iteration.  We fix $\mu$ such that $ \delta < \mu < 0$.
	
	\begin{proposition} \label{W1}
		$\| u \|_{C^0_{\mu}} \leq C$.
	\end{proposition}
	Let $\psi$ be a smooth convex function of one real variable with $\psi(t) =1$ for $t \leq 1$ and $\psi(t) = t$ for $t \geq 2$. Recall that $r$ is the intrinsic distance in the flat metric to $0$ and define $\rho = \psi \circ r$.  In order to prove Proposition \ref{W1}  we introduce the norm
	$$ \| u \|_{L^p_{\mu}}^p = \int_{\mathbb{C}^2} |u|^p \rho^{-p\mu} \rho^{-4} \omega_0^2 . $$
	Because $ u \in C^0_{\delta}$ and $\delta < \mu$ we have that $u \in C^0_{\mu}$, $u \in L^p_{\mu}$ for all $p \geq 1$ and $\|u\|_{C^0_{\mu}} = \lim_{p \to \infty} \| u \|_{L^p_{\mu}}$.
	
	\begin{lemma}
		For $p \geq 2$, $p \mu \leq -2$ we have
		\begin{equation} \label{induction2}
		\| u \|^p_{L^{2p}_{\mu}} \leq Cp \left( \| u\|^{p-1}_{L^{p-1}_{\mu}} + \|u\|^p_{L^p_{\mu}} \right) .
		\end{equation}
	\end{lemma}
	This lemma corresponds to Proposition 8.6.8 in \cite{Joyce}.  \ref{induction2} is a weighted version of inequality \ref{induction}. To prove the lemma we have to use the Sobolev inequality for the metric $\omega_0$ together with an integration by parts argument. We refer to pages 190-192 in \cite{Joyce}. The only work we have to do is to check that the relevant integration by parts hold in our context of metrics with cone singularities -as we did in the proof of Proposition \ref{C0EST}-; this is straightforward and we omit the details.
	To prove Proposition \ref{W1} we note that if $p_0 = (-4/ \mu)$, then $\| u \|_{L^{p_0}_{\mu}} = \| u \|_{L^{p_0}}$ and we already have a bound on this quantity. Finally, an induction argument using \ref{induction2} gives the desired bound on $\|u\|_{C^0_{\mu}}$.
	
	We move on to  state our second result, Proposition \ref{W2}. The proof uses the linear theory we developed and follows the lines of pages 193-195 in \cite{Joyce}. We pause for a moment to touch on a technical point: There is a mistake in the definition of the weighted $C^{\alpha}$ semi-norm given by formula 8.6, page 179 of \cite{Joyce}. The problem is that the semi-norm only compares points which are at distance less than the injectivity radius. This forbids the use of the interior Schauder estimates and scaling arguments needed to establish the linear theory, see the proof of  Lemma \ref{Est 1}. The arguments in pages 193-195 of \cite{Joyce} deal with this wrong semi-norm; nevertheless it is not hard to adapt the arguments to prove what we need.
	
	\begin{proposition} \label{W2}
		$ \| u \|_{2, \alpha, \delta} \leq C$
	\end{proposition}
	
	\begin{proof}
		Write $\omega_u^2 = e^f \omega_0^2$  as $ i \partial \overline{\partial} u \wedge (2\omega_0 + i \partial \overline{\partial} u ) = (e^f -1) \omega_0^2$.  We get
		\begin{equation} \label{EQ1}
		\triangle_0 u = (e^f -1) + \psi , 
		\end{equation}
		with $\psi = u_{i\overline{j}} ^2$. We could also have written
		\begin{equation} \label{EQ2}
		\triangle u = H (e^f-1), 
		\end{equation}
		where $\triangle$ is the Laplace operator of the metric $\omega_{u/2} = \omega_0 + i \partial \overline{\partial} (u/2)$ and $H = \omega_{u/2}^2 / \omega_0^2 $.
		Since $\omega_{u/2} = (1/2) \omega_0 + (1/2) \omega_u \geq (1/2) \omega_0$ and we have a bound on the $C^{2, \alpha}$ norm of  $u$, we conclude that
		\begin{equation} \label{int1}
		\| u \|_{C^{2, \alpha} (B_1(x))} \leq C \left( \| \triangle u \|_{C^{2, \alpha} (B_2(x))} + \| u \|_{C^0(B_2(x))} \right) , 
		\end{equation}
		with a constant $C$ independent of $x$. We multiply \ref{int1} by $\rho(x)^{-\mu}$ to get
		\begin{equation} \label{bound1}
		\| u_{i\overline{j}} \|_{0, \mu} \leq C, \hspace{3mm} \rho(x)^{-\mu} \frac{|u_{i\overline{j}} (x) - u_{i\overline{j}} (y)|}{d(x, y)^{\alpha}} \leq C \hspace{2mm} \mbox{whenever} \hspace{2mm} d(x, y) < 1 .
		\end{equation}
		
		Take $\mu < \tilde{\mu} < 0$, $\tilde{\mu} = \mu + \alpha$ such that $ 2\tilde{\mu} < -2$. At this point we impose some restrictions on the choices of $\delta$, $\mu$, $\tilde{\mu}$.  We start with $-2 < \delta < -1 -\alpha$, then we take $ \delta < \mu < -1 -\alpha$ and $\tilde{\mu} < -1$. We claim that \ref{bound1} implies that $ \| u_{i\overline{j}} \|_{\alpha, \tilde{\mu}} \leq C$. In fact one only needs to consider the case of $d(x, y) \geq 1$, let's say that $\rho(x) \leq\rho(y)$ and estimate 	
		$$ \rho(x)^{-\tilde{\mu} + \alpha} \frac{|u_{i\overline{j}} (x) - u_{i\overline{j}} (y)|}{d(x, y)^{\alpha}} \leq \rho(x)^{-\mu} (|u_{i\overline{j}} (x) +| u_{i\overline{j}}(y)|) \leq 2C . $$
		We use \ref{EQ1} and Proposition \ref{LTHRM} to conclude that $\| u \|_{2, \alpha, 2 + 2\tilde{\mu}} \leq C$. Then $\|u_{i\overline{j}} \|_{\alpha, 2\tilde{\mu}} \leq C$, so that $\| \psi \|_{\alpha, 4\tilde{\mu}} \leq$. Since $4\tilde{\mu} < -4 < \delta -2$,  we can use  \ref{EQ1} and Proposition \ref{LTHRM} again to obtain $\|u\|_{2, \alpha, \delta} \leq C$.
		
	\end{proof}
	
	\subsection{Proof of Theorem \ref{thm}} \label{PFTHM}

	\begin{proof}
		Let $\omega_0$ be the metric given by Proposition \ref{BR}. Take any $0 < \alpha < \beta^{-1}-1$,  $\delta \in (-2, 0)$ and consider the set 
		\begin{equation} 
		T = \lbrace t \in [0,1] : \exists u_t \in C^{2, \alpha}_{\delta} \hspace{3mm}\mbox{solution of} \hspace{3mm} (\omega_0 + i \partial \overline{\partial} u_t)^2 = e^{tf_0} \omega_0^2 \rbrace . 
		\end{equation}
		Clearly $0 \in T$, with $u_0 =0$. Note that if $t \in T$, then $\omega_t= \omega_0 + i \partial\overline{\partial}u_t$ is a K\"ahler metric with cone angle $2\pi\beta$ along $C$. The positivity follows from the equation, the decay of $i\partial\overline{\partial} u_t$ and the connectedness of $\mathbb{C}^2 \setminus C$.  Proposition \ref{LTHRM} and remark \ref{LRMK}, together with the implicit function theorem, imply that the set $T$ is open. Proposition \ref{a priori est} gives us that the set $T$ is closed. We set $\omega_{RF} = \omega_1$. It is easy to check that $\omega_{RF}$ has the desired properties. It follows from this proof that we can improve the statement on the asymptotic behavior and we can say that, outside a compact set,  $\omega_{RF} - H^{*}\omega_F \in C^{\alpha}_{\gamma}$ for any $\gamma > \max \lbrace -1/c, -4 \rbrace$.
		
	\end{proof}

\bibliographystyle{plain}
\bibliography{References}

\begin{thebibliography}{10}

\bibitem{Anderson}
Michael~T. Anderson.
\newblock Ricci curvature bounds and {E}instein metrics on compact manifolds.
\newblock {\em J. Amer. Math. Soc.}, 2(3):455--490, 1989.

\bibitem{Bartnik}
Robert Bartnik.
\newblock The mass of an asymptotically flat manifold.
\newblock {\em Comm. Pure Appl. Math.}, 39(5):661--693, 1986.

\bibitem{Brendle}
Simon Brendle.
\newblock Ricci flat {K}\"ahler metrics with edge singularities.
\newblock {\em Int. Math. Res. Not. IMRN}, (24):5727--5766, 2013.

\bibitem{CDSII}
Xiuxiong Chen, Simon Donaldson, and Song Sun.
\newblock K\"ahler-{E}instein metrics on {F}ano manifolds. {II}: {L}imits with
  cone angle less than {$2\pi$}.
\newblock {\em J. Amer. Math. Soc.}, 28(1):199--234, 2015.

\bibitem{ChenWang2}
Xiuxiong Chen and Yuanqi Wang.
\newblock On the long time behaviour of the conical {K}{\"a}hler-{R}icci flows.
\newblock {\em Journal f{\"u}r die reine und angewandte {M}athematik (Crelles
  Journal)}, 2014.

\bibitem{ChenWang}
Xiuxiong Chen and Yuanqi Wang.
\newblock {$C^{2,\alpha}$}-estimate for {M}onge-{A}mp\`ere equations with
  {H}\"older-continuous right hand side.
\newblock {\em Ann. Global Anal. Geom.}, 49(2):195--204, 2016.

\bibitem{ConlonHein}
Ronan~J. Conlon and Hans-Joachim Hein.
\newblock Asymptotically conical {C}alabi-{Y}au manifolds, {I}.
\newblock {\em Duke Math. J.}, 162(15):2855--2902, 2013.

\bibitem{Donaldson}
S.~K. Donaldson.
\newblock K\"ahler metrics with cone singularities along a divisor.
\newblock In {\em Essays in mathematics and its applications}, pages 49--79.
  Springer, Heidelberg, 2012.

\bibitem{GilbargTrudinger}
David Gilbarg and Neil~S. Trudinger.
\newblock {\em Elliptic partial differential equations of second order}.
\newblock Classics in Mathematics. Springer-Verlag, Berlin, 2001.
\newblock Reprint of the 1998 edition.

\bibitem{GriffithsHarris}
Phillip Griffiths and Joseph Harris.
\newblock {\em Principles of algebraic geometry}.
\newblock Wiley Classics Library. John Wiley \& Sons, Inc., New York, 1994.
\newblock Reprint of the 1978 original.

\bibitem{JMR}
Thalia Jeffres, Rafe Mazzeo, and Yanir~A. Rubinstein.
\newblock K\"ahler-{E}instein metrics with edge singularities.
\newblock {\em Ann. of Math. (2)}, 183(1):95--176, 2016.

\bibitem{Jeffres}
Thalia~D. Jeffres.
\newblock Uniqueness of {K}\"ahler-{E}instein cone metrics.
\newblock {\em Publ. Mat.}, 44(2):437--448, 2000.

\bibitem{Joyce}
Dominic~D. Joyce.
\newblock {\em Compact manifolds with special holonomy}.
\newblock Oxford Mathematical Monographs. Oxford University Press, Oxford,
  2000.

\bibitem{Kronheimer}
P.~B. Kronheimer.
\newblock The construction of {ALE} spaces as hyper-{K}\"ahler quotients.
\newblock {\em J. Differential Geom.}, 29(3):665--683, 1989.

\bibitem{LuoTian}
Feng Luo and Gang Tian.
\newblock Liouville equation and spherical convex polytopes.
\newblock {\em Proc. Amer. Math. Soc.}, 116(4):1119--1129, 1992.

\bibitem{MondelloPanov}
Gabriele Mondello and Dmitri Panov.
\newblock Spherical {M}etrics with {C}onical {S}ingularities on a 2-{S}phere:
  {A}ngle {C}onstraints.
\newblock {\em Int. Math. Res. Not. IMRN}, (16):4937--4995, 2016.

\bibitem{PacardRiviere}
Frank Pacard and Tristan Rivi\`ere.
\newblock {\em Linear and nonlinear aspects of vortices}, volume~39 of {\em
  Progress in Nonlinear Differential Equations and their Applications}.
\newblock Birkh\"auser Boston, Inc., Boston, MA, 2000.
\newblock The Ginzburg-Landau model.

\bibitem{Panov}
Dmitri Panov.
\newblock Polyhedral {K}\"ahler manifolds.
\newblock {\em Geom. Topol.}, 13(4):2205--2252, 2009.

\bibitem{Rubinstein}
Yanir~A. Rubinstein.
\newblock Smooth and singular {K}\"ahler-{E}instein metrics.
\newblock In {\em Geometric and spectral analysis}, volume 630 of {\em Contemp.
  Math.}, pages 45--138. Amer. Math. Soc., Providence, RI, 2014.

\bibitem{Simon}
Leon Simon.
\newblock {\em Theorems on regularity and singularity of energy minimizing
  maps}.
\newblock Lectures in Mathematics ETH Z\"urich. Birkh\"auser Verlag, Basel,
  1996.
\newblock Based on lecture notes by Norbert Hungerb\"uhler.

\bibitem{Szekelyhidi}
G\'abor Sz\'ekelyhidi.
\newblock {\em An introduction to extremal {K}\"ahler metrics}, volume 152 of
  {\em Graduate Studies in Mathematics}.
\newblock American Mathematical Society, Providence, RI, 2014.

\bibitem{Troyanov2}
Marc Troyanov.
\newblock Metrics of constant curvature on a sphere with two conical
  singularities.
\newblock In {\em Differential geometry ({P}e\~n\'\i scola, 1988)}, volume 1410
  of {\em Lecture Notes in Math.}, pages 296--306. Springer, Berlin, 1989.

\bibitem{Troyanov1}
Marc Troyanov.
\newblock Prescribing curvature on compact surfaces with conical singularities.
\newblock {\em Trans. Amer. Math. Soc.}, 324(2):793--821, 1991.

\bibitem{Yau}
Shing~Tung Yau.
\newblock On the {R}icci curvature of a compact {K}\"ahler manifold and the
  complex {M}onge-{A}mp\`ere equation. {I}.
\newblock {\em Comm. Pure Appl. Math.}, 31(3):339--411, 1978.

\end{thebibliography}

\end{document}